\documentclass[11pt]{amsart}
\usepackage{enumitem}  
\usepackage{csquotes}	 
\usepackage[T1]{fontenc} 
\DeclareUnicodeCharacter{2212}{-}
\usepackage[english]{babel} 
\usepackage{tabu} 
\usepackage{ae, aecompl} 
\usepackage{stmaryrd} 
\usepackage{appendix} 
\usepackage{scrextend} 
\usepackage{amsmath}
\usepackage{amssymb}
\usepackage{amsfonts}
\usepackage{amsthm}  
\usepackage{bm} 
\usepackage{latexsym}
\usepackage{graphicx}    
\usepackage{tikz} 
\usepackage{color} 
\usepackage{mathdots}
\usepackage{todonotes} 
\usepackage{mathrsfs} 
\usepackage{verbatim} 
\usepackage{mathtools} 
\usepackage{hyperref}

\newtheorem{thm}{Theorem}[section]
\newtheorem*{thm*}{Theorem} 
\newtheorem{lem}[thm]{Lemma}  

\newtheorem{prop}[thm]{Proposition}
\theoremstyle{definition}
\newtheorem{defn}[thm]{Definition}

\theoremstyle{remark}
\newtheorem{rmk}[thm]{Remark}

\numberwithin{equation}{section}
\numberwithin{figure}{section}

\DeclareMathAlphabet{\matholdcal}{OMS}{cmsy}{m}{n}

\newcommand{\ip}[2]{\langle #1 , #2 \rangle}    

\newcommand\tr{\mbox{tr}\,}   
\newcommand\C{\mathbb C}    
\newcommand\R{\mathbb R}    
\newcommand\Oc{\mathbb O}    
\newcommand\h{\mathbb H} 
\newcommand\style{\mathfrak}
\newcommand\GL{\mathrm{GL}} 
\newcommand\SL{\mathrm{SL}} 
\newcommand\G{\mathbf{G}} 
\newcommand\Spin{\mathrm{Spin}} 
\newcommand\so{\mathrm{SO}} 
\newcommand\prm{\mathfrak p_F} 
\newcommand\of{\matholdcal O_F} 
\newcommand{\ol}{\overline} 
\newcommand{\adj}[1]{\operatorname{adj}(#1)}
\DeclareMathOperator{\ran}{Im} 
\DeclareMathOperator{\aut}{Aut} 
\DeclareMathOperator{\id}{Id} 
\DeclareMathOperator{\stab}{Stab} 
\DeclareMathOperator{\Int}{Int}  
\DeclareMathOperator{\Ad}{Ad}  
\DeclareMathOperator{\ind}{Ind} 
\DeclareMathOperator{\Hom}{Hom} 
 
\begin{document}
\title[Representations of $G_2$ distinguished by $\mathrm{SO}_4$]{Parabolically induced representations of $p$-adic $G_2$ distinguished by $\mathrm{SO}_4$, I}

\author[S.~Dijols]{Sarah Dijols}
\address{Mathematical Sciences 566, University of Calgary, Calgary, Alberta T2N 1N4, Canada}
\email{sarah.dijols@hotmail.fr}
\thanks{}

\date{\today}
\dedicatory{}
\begin{abstract}
We consider the parabolically induced representations of the symmetric space $\so_4\backslash G_2$ over a p-adic field using the geometric lemma when the inducing parabolic is $P_{\beta}$. Using an explicit description of the embedding of $G_2$ in $GL_8$, we characterize precisely the induced representations which are $(\so_4, \chi)$-distinguished, given a certain type of involutions is chosen.
\end{abstract}
\maketitle
\setcounter{tocdepth}{1}
\tableofcontents

\section{Introduction}\label{sec-intro}

The aim of this paper is to identify certain parabolically induced complex representations of the exceptional group $\G_2(F)$, over a $p$-adic field $F$, that admit a linear functional invariant under the special orthogonal group $\so_4(F)$. 

In the last two decades, motivated by the study of period integrals, many works \cite{FLO12, Off06, arnabO} have described the distinguished representations of various classical groups, for instance the general linear groups and the unitary groups by their symplectic, unitary or general linear subgroups. Around the same time, the far-reaching Sakellaridis-Venkatesh conjectures have reignited interest and gave further motivations in the description and classification of representations of $p$-adic symmetric spaces (as a particular instance of spherical varieties) $G \slash H$.
\newline

In this realm of research, very little has been understood regarding exceptional groups, a recent work of Gan-Gomez \cite{gangomez}, dealt with many low-rank varieties, including $\G_2 \slash \SL_3$ (a spherical variety which is not a symmetric space). Their work, however, does not deal with a precise description or classification of representations of $\G_2(F)$ which might be distinguished by $\SL_3(F)$. Indeed, such classification would require to use the geometric lemma method (also known as "orbit method" since it relies on analysing consecutively a set of parabolic orbits). Our paper constitutes the first instance of its application (implementable if the quotient is a symmetric space, or in the Galois case) to an exceptional group. The main tools in our investigation have been exposed in \cite{offen2017}. The drawback of our approach is that it allows us to \emph{only} deal with parabolically induced representations. 
\newline

The strategy described in the paper of Offen \cite{offen2017} consists in reducing the question of distinction of the induced representations of $G$ by $H$ to a question of distinction at the level of a subgroup $L_x \subseteq M$ associated to a representative $x$ for each orbit. It involves computing the relevant subgroups $Q_x= L_x\rtimes U_x$ and associated modular character. To do so, since none of the patterns of classical groups were reproducible in our context, we have used the mathematical software \textbf{SageMath} and an explicit embedding of $\G_2$ into $\GL_8$. In this paper, we deal with the case of the parabolic $P_{\beta}$, and leave the case of $P_{\alpha}, B$ to a subsequent paper. Following a method of \cite{Ginzburg}, we first find eleven double cosets' representatives and therefore as many parabolic orbits to be studied. To implement the subsequent steps, we need to write an explicit expression of the Levi subgroup $M_{\beta}$ (using Bruhat cells), identify the admissible orbits, and verify their closedness or openess. We have also identified the matching elements in $W_{\beta}\backslash W\slash W_{\beta}$ for each orbit representative. Finally, the results where we identified the Levi subgroups $L= M \cap w_x M w_x^{-1}$, for each matching element $w_x$, and also implemented various computations to check properties of the orbits (see in \cite{offen2017}) using a modified version of the orbit representatives have been included in the form of codes (see also the Appendix). A better strategy was eventually found using a stricter definition of admissibility (see Definition \ref{strictadm}), already offered in the literature \cite{offenluminy}. 

Our main results are the following: 
\begin{thm}[Closed orbit] \label{closedorbit}
Let $\chi$ be a character of $\so_4(F)$. It is a quadratic character of $F^{\times}$. It can be seen as a character of $\GL_2$ (those are given by $\chi \circ \det$ for a quasi-character $\chi$ of $F^{\times}$). Let $P_{\beta}$ denote the maximal parabolic corresponding to the root $\beta$. The parabolic induced representations of $G_2$ which are $(\so_4,\chi)$-distinguished include the following representations:
\begin{itemize}
\item The induction from $P_{\beta}$ to $G_2$ of the reducible principal series of $\GL_2$, $I(\chi\delta_{P_{\beta}}^{1/2}|.|^{-1/2}\otimes |.|)$, and the induction of the direct sum of this reducible principal series with any smooth representation of $\GL_2$.
\item The induced representation $I_{P_{\beta}}^{G_2}(\chi\delta_{P_{\beta}}^{1/2})$. 
\item The induced representation $I_{P_{\beta}}^{G_2}(\sigma)$ with $\sigma$ a cuspidal representation of $\GL_2$ whose central character is equal to $(\chi\delta_{P_{\beta}}^{1/2})|_Z$ or the induced representation of the direct sum of such representation with any smooth representation of $\GL_2$.
\item The induced representation $I_{P_{\beta}}^{G_2}(\tau \oplus \chi\delta_{P_{\beta}}^{1/2})$ for $\tau$ any smooth representation of $\GL_2$.
\end{itemize}
\end{thm}

\begin{thm*}[Distinguished induced parabolic representations and admissible orbits]
We take the involution $\theta$ defining $\so_4(F) = \G_2^{\theta}(F)$ to be of the form $\theta_{t_i}$ for $i \in \{0,1,2\}$ as defined in the Subsection \ref{convention}, $M_{\beta}$ as defined in Equation \ref{mb}.
The parabolically induced representations from the parabolic $P_{\beta}$ of $\G_2(F)$ distinguished by $\so_4(F)$ whose linear forms arise from admissible, open or closed, orbits are necessarily of the form given in the previous Theorem \ref{closedorbit}.
\end{thm*}

Our computations also reveal a mysterious and exciting phenomenon with the open orbits which are parametrized by the number of quadratic extensions $E$ of $F$, see the Proposition \ref{amazing}.
\newline
The context of dealing with the split exceptional group $\G_2$ gives to this paper its computational (via \textit{SageMath}) nature. All our codes and \textit{SageMath} computation are available at the following link: \url{https://github.com/sarahdijols/G2SO4}. It is worth mentioning that this software helps us only to multiply many 8-dimensional matrices, but all these multiplications could be, in principle, done by hand. No programming skills are needed to understand the codes available at this link. A byproduct of the strategy developed in this work is to provide explicit expressions of the tori, roots subgroups, and Levi subgroups of $G_2$, which allow, for instance, to compute the modulus for the maximal parabolic subgroups and the Borel of $G_2$. We believe these codes could be useful to the math community.
\newline

Here we briefly outline the contents of the paper. 
In Section \ref{sec-note}, we establish notation and recall some basic definitions.  
Section \ref{sec-pblc-dist} contains a review of two key results proved by Offen in \cite{offen2017}, and Section  \ref{sec-data} provides some result on the distinguished representations that form the inducting data for the representations of $\G_2(F)$ studied here.
We study the structure of the symmetric space $\G_2(F) / \so_4(F)$ in Section \ref{sec-g2}; additional detail is provided in Appendix \ref{sec-g2-apdx}, where an embedding of $\G_2(F)$ into $\GL_8(F)$ is discussed.
In Section \ref{doublecosets}, we describe the double cosets and double cosets representatives, while in Section \ref{analysisorb}, we study the orbits in $\G_2(F) / \so_4(F)$ under the twisted action of standard parabolic subgroups of $\G_2(F)$. Finally, the main results on $\so_4(F)$-distinguished parabolically induced representations of $\G_2(F)$ are stated and proved in Section \ref{sec-results}.
\newline

We, finally, mention here that this paper considers only sufficient conditions for distinction, as presented in the Propositions 7.1 and 7.2 of \cite{offen2017}. The necessary conditions which may involve using Proposition 4.1 in \cite{offen2017} will be addressed, to the greatest extent possible, in our subsequent work. The reader will notice that all the ingredients have been prepared to do so in the form of codes, as the algorithm to compute the expressions for the subgroups $L_x \subset M$ and the relevant modular characters have been written and tested (see the files "delta-functions-Pb-Pa-min.ipynb" and "delta-functions-Q-x-clean(1).ipynb" in particular).

\subsection*{Acknowledgements}
The author would like to express her gratitude to Steven Spallone for communicating, in its entirety, the material in Appendix \ref{sec-g2-apdx} on the structure of $\G_2$ as the group of automorphisms of the Cayley algebra. She also thanks Qing Zhang and David Ginzburg for giving her a few precious references and hints on $\G_2$, and Dipendra Prasad for some general comments. She would like to warmly thank Arnab Mitra and Jerrod Smith for many discussions on the method and on the earlier calculations done in this work, and Jerrod Smith for his contribution to a few files in the code, and to the organization of this paper. The author has benefited from very good work conditions at the YMSC, Beijing, and at the MPIM, Bonn.

\section{Notation and Preliminaries}\label{sec-note}
Let $F$ be a non-Archimedean local field of characteristic zero and odd residual characteristic. Let $\of$ be the ring of integers of $F$ with prime ideal $\prm$.
 Fix a uniformizer $\varpi$ of $F$; note that $\prm = \varpi \of$. Let $q$ be the cardinality of the residue field $k_F = \of / \prm$.  Let $|\cdot |_F$ denote the normalized absolute value on $F$ such that $|\varpi|_F = q^{-1}$.  We write $|\cdot|$ for the usual absolute value on the field $\C$ of complex numbers.  

Let $G=\G(F)$ be the $F$-points of a connected reductive group defined over $F$. We let $e$ denote the identity element of $G$. For any $g \in G$, we denote the inner $F$-automorphism of $G$ given by conjugation by $g$ by $\Int_g$.  That is, $\Int_g(x) = g x g^{-1}$ for all $x \in G$.  Recall that the map $\Int : G \rightarrow \aut_F(G)$ given by  $g \mapsto \Int_g$ is a group homomorphism.  Moreover, $\ker(\Int) = Z_G$ is the centre of $G$.  Note that if  $g^2 = e$, then $\Int_g$ is an involution, that is, an order two automorphism.  Indeed, if $g^2 = e$, then for any $x \in G$
\[
(\Int_g)^2(x) = \Int_g \circ \Int_g(x) = g^2 x g^{-2} = exe = x,
\]
and $(\Int_g)^2 = \id_G$ is the identity map on $G$.  
Observe that $x\in G$ is fixed by $\Int_g$ if and only if $x \in C_G(g)$, where $C_G(g)$ is the centralizer of $g$ in $G$.

All representations are over complex vector spaces.  We will often abuse notation and refer to a representation $(\pi,V)$ of $G$ simply as $\pi$.
We write $\bm{1}_G:G \rightarrow \C^\times$ for the trivial character of $G$, that is, $\bm{1}_G(g) = 1$ for all $g \in G$.
We assume that all representations $(\pi,V)$ of $G$ are smooth in the sense that the stabilizer of any vector $v \in V$ is an open subgroup of $G$.
A character of $G$ is a one-dimensional smooth representation of $G$ (not necessarily unitary).

Let $P$ be a parabolic subgroup of $G$.  Let $N$ be the unipotent radical of $P$, and let $M$ be a Levi subgroup of $P$.  Let $\delta_P : P \rightarrow \R_{>0}$ be the modular character of $P$.  Recall that $\delta_P(p) = |\det \Ad_{\style n}(p)|_F$ for all $p \in P$, where $\Ad_{\style n}$ denotes the adjoint action of $P$ on the Lie algebra $\style n$ of $N$ \cite{Casselman-book}.
Given a smooth representation $(\sigma,W)$ of $M$, we denote the normalized parabolic induction of $\sigma$ along $P$ by $I_P^G(\sigma) = \ind_P^G(\delta_P^{1/2}\otimes\sigma)$.

\subsection{Distinguished representations}
Let $H$ be a closed subgroup of $G$, and let $\chi$ be a character of $H$.
Let $(\pi,V)$ be a smooth representation of $G$. 

\begin{defn}
The representation $(\pi,V)$ is said to be $(H,\chi)$-distinguished if there exists a nonzero linear functional $\lambda$ in $\Hom_H(\pi,\chi)$.
\end{defn}

If $(\pi,V)$ is $(H,\bm{1}_H)$-distinguished, then we will simply say that $(\pi,V)$ is $H$-distinguished.  The $H$-distinguished representations of $G$ are precisely those representations of $G$ that are relevant to the study of harmonic analysis on the quotient $G/H$. Indeed, given a nonzero $H$-invariant linear functional $\lambda$ in  $\Hom_H(\pi,\bm{1}_H)$ the linear transformation sending $v \in V$ to the function $\varphi_{\lambda,v}$, where $\varphi_{\lambda,v}(g) = \ip{\lambda}{\pi(g^{-1})v}$ for all $g\in G$, defines an intertwining operator from $(\pi,V)$ to the regular representation of $G$ on the smooth complex valued functions on $G/H$.  Moreover, any such intertwining operator arises this way.
In studying distinguished parabolically induced representations it is necessary to consider $(H,\chi)$-distinguished representations at the level of the inducing data.

The following elementary result is quite useful.
\begin{lem}\label{lem-dist-central-char}
Let $(\pi,V)$ be a representation of $G$.  Suppose that $\pi$ admits a central character $\omega_\pi$.
Let $\chi$ be a character of $H$. 
If $\pi$ is $(H,\chi)$-distinguished, then $\chi \rvert_{H\cap Z} = \omega_\pi \rvert_{H\cap Z}$.
\end{lem}

\begin{proof}
Since $\pi$ is $(H,\chi)$-distinguished, there exists a nonzero linear functional $\lambda$ in $\Hom_H(\pi,\chi)$.  Let $v \in V$ so that $\ip{\lambda}{v} \neq 0$.  Suppose that $z \in H \cap Z$.
Then since $\lambda$ is $H$-invariant and $\pi$ has central character $\omega_\pi$ we have that
\begin{align*}
\chi(z)\ip{\lambda}{v} = \ip{\lambda}{\pi(z)v} = \ip{\lambda}{\omega_{\pi}(z) v} = \omega_{\pi}(z) \ip{\lambda}{v}.
\end{align*}
Therefore,
\begin{align*}
0 = (\chi(z) - \omega_\pi(z)) \ip{\lambda}{v}
\end{align*}
and since $\ip{\lambda}{v}\neq 0$ it follows that $\chi(z) = \omega_\pi(z)$.  Therefore, the restriction $\chi \rvert_{H\cap Z}$ of $\chi$ to $H\cap Z$ agrees with the restricted central character $\omega_\pi \rvert_{H\cap Z}$.
\end{proof}

\section{Distinction for parabolically induced representations}\label{sec-pblc-dist}
Here we recall the general results of Offen \cite{offen2017} that we utilize below.  We use mostly the same notations as Offen.

Let $G=\G(F)$ be the $F$-points of a connected reductive group $\G$ defined over $F$.  Let $\theta$ be an $F$-rational involution of $\G$.  Let $H = \G^\theta(F)$ be the $F$-points of the $\theta$-fixed set $\G^\theta$ in $\G$.  
Let $X = \{ g \in G : g \theta(g) = e \}$.  Elements of the set $X$ are referred to as the $\theta$-split elements in $G$.
  The set $X$ carries a $G$-action given by
\begin{align*}
(g,x) & \mapsto g\cdot x = g x \theta(g)^{-1}
\end{align*}
for all $g\in G$ and $x \in X$.
Of course, $e\theta(e) = e$, so the identity element of $G$ lies in $X$.
The stabilizer of $e$ under the $G$-action on $X$ is the subgroup $H$ of $\theta$-fixed points.
It follows that the map $G \rightarrow X$ given by $g \mapsto g \cdot e$ defines an embedding of the symmetric space $G / H$ in $X$ as the $G$-orbit of the identity.

Let $x\in X$ be a $\theta$-split element of $G$.  The $F$-rational automorphism $\theta_x$ of $G$ defined by
\begin{align*}
\theta_x(g) & = x\theta(g)x^{-1} & \text{for all } \, g \in G
\end{align*}
is an involution.  For any subgroup $K$ of $G$ let $K_x = \stab_K(x)$ be the stabilizer of $x$ in $K$ for the $G$ action on $X$.  Then $H = G_e$ and $K_x = K^{\theta_x}$ for any subgroup $K$ of $G$ and $x \in X$; however,  $K$ need not be $\theta_x$-stable so it is convenient to note that $K_x = (K \cap \theta_x(K))^{\theta_x}$.

We will assume that $\G$ is split over $F$.  Let $B$ be a Borel subgroup of $G$ with unipotent radical $N$.  By \cite[Lemma 2.4]{helminckwang1993} there exists a $\theta$-stable maximal $F$-split torus $T$ of $G$ contained in $B$.  We have that $B = TN$.
A parabolic subgroup $P$ of $G$ is standard if it contains the Borel subgroup $B$.
Suppose that $P$ is a standard parabolic subgroup of $G$, then $P$ admits a unique Levi subgroup $M$ that contains $T$.
Let $U$ be the unipotent radical of $P$.  We will always work with a standard Levi factorization $P = MU$ with $T \subseteq M$.
Let $N_{G,\theta}(M) = \{ g \in G : M = g\theta(M) g^{-1} \}$.

Let $\chi$ be a character of $H$ and let $\eta \in G$.  Write $\chi^{\eta^{-1}}$ for the character of $\eta^{-1}H\eta$ given by $\chi^{\eta^{-1}}(h') = \chi (\eta h' \eta^{-1})$ for all $h' \in \eta^{-1}H\eta$.
\newline

The following proposition deals with the case of a closed orbit.
\begin{prop}[For instance Proposition 7.1 in \cite{offen2017}] \label{offen7.1}
Let $\chi$ be a character of $H$.
Let $P = MU$ be a standard parabolic subgroup of $G$ with unipotent radical $U$ and Levi factor $M$. Let $(\sigma, W)$ be a smooth representation of $M$.
Suppose that $\eta \in G$ so that $x = \eta \cdot e \in N_{G,\theta}(M)$ and $\theta_x(P) = P$.  
If $\sigma$ is $(M_x, \delta_{P_x}\delta_P^{-1/2}\chi^{\eta^{-1}})$-distinguished, then $I_P^G(\sigma)$ is $(H,\chi)$-distinguished.  
\end{prop}

Let us mention that an analogous Proposition for the open orbit is given in \cite{offen2017}.
When applying the above results in Section \ref{sec-results}, it will be important for us to carefully choose representatives for the various $P$-orbits in $X$ following \cite[Section 3]{offen2017}.  We discuss the parabolic orbits in the setting of $G = \G_2(F)$ and $H=\so_4(F)$ in Section \ref{sec-adm-orb}.

Finally, let us recall here a recent related result of Prasad in \cite{genericprasad} which assures us of the existence of a generic unitary principal series representation of $\G_2(F)$ distinguished by $\so_4(F)$.
 
\begin{prop}[Proposition 11 in \cite{genericprasad}]
Let $(G, \theta)$ be a symmetric space over a finite or a non-Archimedean local field $k$ which is quasi-split over $k$, thus there is a Borel subgroup $B$ of $G$ over $k$ with $B \cap \theta(B) = T$, a maximal torus of $G$ over $k$. If $k$ is finite, assume that its cardinality is large enough (for a given $G$). Then there is an irreducible generic unitary principal series representation of $G(k)$ distinguished by $G^{\theta}(k)$.
\end{prop}

\subsection{The admissibility condition} \label{admcondition}

Let us recall from \cite{offen2017} the existence of a map from the set of parabolic orbits to the set of twisted involutions in the Weyl group, which is, in general, neither injective nor surjective: 
$$\iota_M : P\backslash X \rightarrow {}_{M}{W}{_{M'}}\tau^{-1} \cap \mathcal{S}_0(\theta)$$
Let us notice first that various definitions of \textit{admissibility} have been given in the literature. In \cite{offen2017}, admissibility is given by the following definition:

\begin{defn} \label{dfn1}
We say that $x \in X$ (or $P.x$) is $M$-admissible if $M= w\theta(M)w^{-1}$ where $w= \iota_M(P.x)$.
\end{defn}

Whereas in \cite{offenluminy}[Section 3.2.6], a stricter definition is used:
\begin{defn}[Strict admissibility] \label{strictadm}
$x \in X$ (or $P.x$) is $M$-admissible if $M= x\theta(M)x^{-1}$.
\end{defn}

Possibly, in the context of classical groups these two definitions completely agree, but in our context the set of orbits which are strictly admissible would be larger than the set of admissible orbits. Indeed as computed in the code \enquote{admissibility-with-w}, only $w_0= w_{\alpha} w_{\beta} w_{\alpha} w_{\beta} w_{\alpha}$ among the four-elements set $W_{\beta} \backslash W \slash W_{\beta}$ is likely to be admissible.  
\newline

Let us also remark that this condition is far from subsidiary since a recent work of Offen and Matringe  \cite{matringe2021}, in the case of $p$-adic Galois symmetric spaces, implies that the admissibility condition should be enough for a given orbit to contribute to the distinction of the induced representation space. Their result is expected to be extended to general symmetric spaces. It is therefore important to be able to determine which representatives are $M$-admissible. Notice, however, that our \textit{ad hoc} expression for the Levi $M_{\beta}$ possibly makes this verification a little loose.  
\newline

Finally, to end this section, we add a comment which is best suited here: In \cite{offen2017}, Lemma 3.2 gives us that the representatives $\eta$ of the double coset in $P\backslash G \slash H$ can be chosen so that $x \in Lw$ (for the $w$ as defined in \ref{dfn1}), where $L$ is a standard subgroup of $M$ such that $L=M \cap \theta_x(M)$. Assume this is the case, and let us define $Q$ to be the standard parabolic subgroup of $G$ with standard Levi subgroup $L$ and unipotent radical $V$. \\
Then when we will choose our involution $\theta$ to be $\theta_{t_0}$ (see Section 5 and Proposition \ref{tau} for an explanation of this notation), notice that the conditions $x \in N_{G,\theta}(M)$ in the Propositions \ref{offen7.1} (and its equivalent for open orbit, as presented in \cite{offen2017}) and the equality $M_x= M$ are essentially the same. Furthermore, in this case, $L = M \cap \theta_x(M)$ and therefore $L= M$, so that $\delta_{Q_x} = \delta_{M \ltimes U_x}$. Therefore, in applying both of these propositions, when $\theta= \theta_{t_0}$, we are reduced to the problems of identifying the representations $\sigma \in \hbox{Rep}(M)$ which are distinguished by a certain character of $\GL_2(F)$.

\section{Inducing data}\label{sec-data}
Both of the maximal (proper) parabolic subgroups of $\G_2(F)$ have Levi factor isomorphic to $\GL_2(F)$. In this section, we collect information regarding various distinguished representations of $GL_2=\GL_2(F)$. For representations of $\GL_2(F)$-distinguished by a maximal $F$-split torus, Section 3.1.3 of \cite{offenluminy} provides an excellent summary.

\begin{prop} \label{GL2distinctionofPS}
Let $Z$ be the centre of $GL_2$. A smooth representation $\pi$ of $GL_2$ is $(\GL_2, \chi)$-distinguished if and only if it is one of the following $GL_2$-representations:
\begin{itemize}
\item $\pi$ is isomorphic to $\chi$ (then $\pi$ is irreducible).
\item $\pi$ is a reducible principal series of the form $I(\chi|.|^{1/2}\otimes \chi|.|^{-1/2})$.
\item $\pi$ is a cuspidal representation with central character $\chi|_Z$.
\item $\pi$ is the direct sum of $\tau$ any smooth representation of $GL_2$ with one of the three above representations.
\end{itemize}
In the context of an induced principal series of the maximal Levi $M_{\beta}$ in $G_2$ where $\beta$ is the long root of $G_2$, the induced principal series takes the form $I(\chi|.|^{-1/2}\otimes |.|)$.
\end{prop}

\begin{proof}
A smooth representation $\pi$ of $\GL_2$ is $(\GL_2, \chi)$-distinguished if and only if $\chi$ occurs as a quotient of $\pi$ by a $\GL_2$-subrepresentation, or $\pi$ is a cuspidal representation whose central character equals the character $\chi$ restricted to the center by an application of Lemma \ref{lem-dist-central-char}. Recall any character of $\GL_2$ factors through det. 
\newline
Here, we only justify the second element in the list, the others being obvious. Let us denote $Q(\chi_1, \chi_2)$ the one-dimensional quotient of the reducible principal series $I(\chi_1\otimes \chi_2)$, then $Q(\chi_1, \chi_2) \cong \hbox{Span}\{\chi_0\}$ for $\chi_0$ a quasi-character of $\GL_2(F)$. Notice that $\hbox{Span}\{\chi_0\}$ is $\GL_2$-invariant subspace of $I(\chi_1\otimes\chi_2)$ where $\GL_2(F)$ acts via $\chi_0$ itself.

Given a character $\chi \circ \det: \GL_2(F) \rightarrow \C^{\times}$, where $\chi$ is a quasi-character of $F^{\times}$, by a well-known description of $\GL_2(F)$-representations and reducibility point of principal series (see \cite{bump}, Chapter 4 or \cite{muicreg2} Proposition 1.1) it occurs as an irreducible quotient of a representation of $\GL_2(F)$, if and only if the representation is the reducible principal series $I(\chi_1\otimes\chi_2) = \ind_B^G(\chi|.|^{1/2}\otimes\chi|.|^{-1/2}) = \ind_B^G(\delta_B^{1/2}\chi \otimes \chi)$. 
Finally, the last statement of the proposition is from Proposition 1.1 in \cite{muicreg2}.
\end{proof}

\section{The exceptional group $G_2$ and its symmetric subgroup $\so_4$}\label{sec-g2}

Throughout the rest of this paper unless specified otherwise let $G = \G_2(F)$ be the group of $F$-points of the split exceptional group $\G_2$, and let $H = \so_4(F)$ be the $F$-points the split special orthogonal group $\so_4$. 
We start with a lemma which offers an interesting geometrical interpretation of the subgroup $H$, under certain conditions.

\begin{lem}
Let us assume the characteristic of the field $F$ is different from 2.
Let $\mathcal{C}$ be a composition algebra of dimension 8, $D$ a quaternion subalgebra, $a \in D^{\perp}$, with $N(a) \neq 0$. Assume $N(a)=1$, then the quotient $\G_2 \slash \so_4$ is the space of quaternionic subalgebras of $\mathcal{C}$.
\end{lem}

\begin{proof}
Let $\mathcal{C}$ be a composition algebra (in our context, of dimension 8 over $F$, for instance the octonions, $\Oc$), and $D$ be a finite dimensional composition subalgebra of $\mathcal{C}$. Suppose $a \in D^{\perp}$, with $N(a) \neq 0$ then $D_1 = D\oplus Da$ and $D_1$ is a composition subalgebra. The subalgebra $D_1$ is said to be constructed by doubling from $D$. The norm is given by $N(x+ya) = N(x) - \lambda N(y)$, for $x, y$ in $D$, and $\lambda = -N(a)$. For instance the split octonion (see the Appendix) can be constructed from the split quaternion algebra by such doubling process as in Proposition 1.5.1, \cite{veldkamp}.
\newline
Let now assume this composition $\mathcal{C}$ is an octonion algebra, and $D$ a given quaternion subalgebra. 
If one chooses $a$ to be of norm one, then $\so_4$ is seen as the group $G_{D} = \left\{\sigma \in G = \aut(\mathcal{C}): \sigma(D) = D \right\}$ and the argumentation goes as follows: Since $G$ preserves $D$, it also preserves the orthogonal complement $aD$. if $\sigma \in G_{D}$ acts trivially on $aD$, then $G_{D}$ fixes $a$ so $\sigma(ua)= \sigma(u)a = ua$ and so $\sigma$ acts trivially on $D$ as well, so $\sigma = 1$. 
Thus $G_{D}$ acts faithfully on $aD$ (but not on $D$) and we have an injective homomorphism $G_{D} \hookrightarrow O(4)$.
\newline

It remains to show that $G_{D}$ is of dimension six. To do so one observes that the restriction map from $G_D \rightarrow \aut(D) \cong \so(D_0)(F)$ (here $D_0$ are the trace zero elements in $D$) is surjective by an application of Corollary 1.7.3 in \cite{veldkamp}, and let $K$ be the kernel of this map. Proposition 2.2.1 in \cite{veldkamp} tells us that $\textbf{K}$ (the algebraic group of $\bar{F}$-automorphisms of $\mathcal{C}_{\bar{F}}$ that fix $D_{\bar{F}}$ elementwise) is a 3-dimensional algebraic group and connected. The isomorphism between $D_{\bar{F}}$ and the unitary quaternions inducing those properties induces the same isomorphism at the level of $F$.
Let us remark that the isomorphism $\aut(D) \cong \so(D_0)(F)$ is due to \cite{vigneras}, Theorem I.3.3, and using the fact (see \cite{voight} Corollaries 7.1.2 and 7.1.4 for instance) that every $F$-algebra automorphism of $D$ is inner, i.e $\hbox{Aut}_F(D) \cong D^{\times}\slash F^{\times}$.  
Thus $G_D$ fits inside the exact sequence:
$$1 \rightarrow K \rightarrow G_D \rightarrow \so(D_0) \rightarrow 1 $$
In particular $G_D$ is connected and $\dim G_D = 6$, so $G_D \cong \so_4$. This result is true if $D = \h$ and $D_1 = \Oc$, and holds in a p-adic context with the additional conditions given in the statement of this lemma.
\end{proof}

\begin{rmk}
Notice that in our context, and to embed $\G_2$ into $\GL_8$ (see the Appendix), we have chosen $N(a) = -\lambda = -1$. It would be interesting to consider the embedding of $\G_2$ into $GL_8$ using $N(a)= 1$ and proceed with the remaining steps using this convention. 
\end{rmk}

Let $T$ be a maximal $F$-split torus of $G$.  Let $B$ be a Borel subgroup of $G$ containing $T$ and let $N$ be the unipotent radical of $B$. Then $B = TN$ is a Levi decomposition of $B$.  A parabolic subgroup $P$ of $G$ is standard if it contains the fixed Borel subgroup $B$.  The standard Levi factor $M$ of a standard parabolic $P$ is the unique Levi factor that contains the torus $T$.
Let $W$ be the Weyl group of $G$ defined with respect to $T$.
\newline

Recall that $\G_2$ is simply connected (see \cite[Ch.~24]{Milne2017-book} for instance).
With this fact, one can adjust the results used in the proof of \cite[Lemma 3.2(i)]{cech} to see that all elements of order $2$ in $G$ are conjugate in $G$.  
Moreover, the centralizer of an order-two element in $G$ is isomorphic to $H$.
The two key modifications are to use (1) the fact that the centralizer of a (finite order) semisimple element in a connected group is connected (this is a theorem of Springer and Steinberg, see \cite[Theorem 2.11]{Humphreys1995-book}), and (2) all maximal $F$-split $F$-tori in a smooth connected group are conjugate over the $F$-points of the group (this is a theorem of Borel and Tits, see \cite[Theorem C.2.3]{Pseudo-book}).

Let $\theta = \Int(t_0)$, where $t_0 \in T$ is an order two element (for instance, we can take $t_0 = \gamma(1,-1)$, using the notation of Appendix \ref{sec-g2-apdx}).
Since, $t_0^2 = e$, the inner automorphism $\theta$ is an involution.
Observe that since $t_0 \in T$, the torus $T$ and Borel subgroup $B$ are $\theta$-stable.
The group $G^\theta$ of $F$-points of the $\theta$-fixed points in $G$ is the centralizer of $t_0$ in $G$, and so $G^\theta \cong H$.

\begin{rmk}
Note that $T$ is $\theta$-stable. It follows that $\theta$ induces an involution on the Weyl group $W$	which we also denote by $\theta$.
\end{rmk}

As above, let $X = \{ g \in G : g \theta(g) = e \}$.  Recall that the set $X$ carries a $G$-action given by
\begin{align*}
(g,x) & \mapsto g\cdot x = g x \theta(g)^{-1}
\end{align*}
for all $g\in G$ and $x \in X$.
Of course, $e\theta(e) = e$, so the identity element of $G$ lies in $X$.
The stabilizer of $e \in X$ under the $G$-action is the subgroup $G^\theta$ of $\theta$-fixed points.
The map $G \rightarrow X$ given by $g \mapsto g \cdot e$ defines an embedding of the symmetric space $G / H$ in $X$ as the $G$-orbit of the identity.

\begin{lem} \label{disjointunion}
The set $X$ is disjoint union of two $G$-orbits, namely $G\cdot e$ and the singleton set $\{t_0\}$.
\end{lem}

\begin{proof}
By definition, 	
\begin{align*}
X & = \{ g \in G : g \theta(g) = e \} = \{ g \in G : gt_0g = t_0 \}.
\end{align*}
The $G$-orbit of the identity element is
\begin{align*}
G\cdot e = \{ g\cdot e : g \in G \} 
= \{ gt_0g^{-1}{t_0}^{-1} : g \in G \}. 	
\end{align*}
In particular, for all $g \in G$, $g\cdot e = gt_0g^{-1}{t_0}^{-1} \in X$.
On the other hand, $t_0 \in X$ but $t_0$ is not in $G\cdot e$.
Indeed, since ${t_0}^2 = e$ we have $t_0t_0t_0 = t_0 e = t_0$ so $t_0 \in X$.  Now argue by contradiction and suppose that $t_0 = g\cdot e$ for some $g \in G$.  It follows that
\[
e = {t_0}^2 = (g\cdot e)t_0 = gt_0g^{-1}{t_0}^{-1}t_0 = gt_0g^{-1},
\]
and $t_0 = g^{-1}eg = e$ which contradicts that $t_0 \neq e$ is an order two element of $T$.
Thus, $G\cdot e \cap \{t_0 \} = \varnothing$.  Moreover, $\{t_0\}$ is a $G$-orbit in $X$ because $t_0$ is fixed under the $G$ action on $X$.  Indeed, for any $g \in G$ 
\begin{align*}
g\cdot t_0 = g t_0 \theta(g)^{-1} = g t_0 t_0 g^{-1} {t_0}^{-1} = g e g^{-1} {t_0}^{-1} = {t_0}^{-1} = t_0.
\end{align*}
Finally, we show that $X$ is the union of $G \cdot e$ and $\{t_0\}$.
Suppose that $x \in X$.  Then $xt_0x = t_0$.  Thus
\begin{align*}
(xt_0)^2 = xt_0xt_0 = {t_0}^2 = e.
\end{align*}
Therefore, $xt_0$ is either the identity or an order two element of $G$.
If $xt_0 = e$, then $x = t_0^{-1} = t_0 \in \{t_0\}$.  Otherwise, $xt_0$ has order two and by \cite[Lemma 3.2(i)]{cech} (and the remarks above) $xt_0$ is $G$-conjugate to $t_0$.  In the latter case, there exists $g \in G$ so that $g^{-1} xt_0 g = t_0$, that is, $x = gt_0g^{-1}{t_0}^{-1} = g\cdot e$.
Therefore, $x \in \{t_0\}$ or $x \in G\cdot e$ and $X = G\cdot e \cup \{t_0\}$ is a union of (disjoint) $G$-orbits.
\end{proof}

\subsection{Roots and Weyl groups}
Let $\Delta = \{\alpha, \beta\}$ be a basis of the root system $\Phi$ of $G$ with respect to $T$ where $\alpha$ is the short root and $\beta$ is the long root.
The set of positive roots of $\G_2$ is 
\[
\Phi^+ = \{ \alpha, \beta, \alpha+\beta, 2\alpha+\beta, 3\alpha+\beta, 3\alpha + 2\beta \}.
\]

Let us recall that that we denote $W = N_G(T)\slash T$ the Weyl group of $\G_2$. More generally, for a standard Levi subgroup $M$ of $\G_2$, we denote $W_M=N_M(T)\slash T$ the Weyl group of $M$ with respect to $T$. 
\newline

The Weyl group of $\G_2$ is generated by the simple reflections $w_\alpha$ and $w_\beta$ attached to the roots $\alpha$ and $\beta$.
In particular, $W$ is a finite group of size $12$ and we can realize $W$ as follows:
\begin{align*}
W & = \{e, w_{\alpha}, w_{\beta}, w_{\alpha}w_{\beta}, w_{\beta}w_{\alpha}, w_{\beta}w_{\alpha}w_{\beta}, w_{\alpha}w_{\beta}w_{\alpha}, w_{\beta}w_{\alpha}w_{\beta}w_{\alpha}, w_{\alpha}w_{\beta}w_{\alpha}w_{\beta}, \\
& w_{\alpha}w_{\beta}w_{\alpha}w_{\beta}w_{\alpha}, w_{\beta}w_{\alpha}w_{\beta}w_{\alpha}w_{\beta}, w_{\alpha}w_{\beta}w_{\alpha}w_{\beta}w_{\alpha}w_{\beta}\}.
\end{align*}
We summarize the action of the simple reflections $w_\alpha$ and $w_\beta$ on $\Phi^+$ in Figure \ref{fig-root-act}.
\begin{figure}[h]
\label{fig-root-act}
\caption{Action of $w_\alpha$ and $w_\beta$ on $\Phi^+$}
$$\begin{tabu}{c | c c c c c c} \label{table}
\Phi^+ & \alpha & \beta & \alpha + \beta & 2\alpha + \beta &  3\alpha + \beta & 3\alpha + 2\beta\\ 
\hline 
w_\alpha\cdot \Phi^+ & - \alpha  & 3\alpha + \beta & 2\alpha +\beta & \alpha +\beta & \beta & 3\alpha + 2\beta \\
\hline
w_\beta\cdot \Phi^+  & \alpha + \beta & -\beta & \alpha & 2\alpha + \beta & 3\alpha + 2\beta & 3\alpha + \beta \\   
\end{tabu}$$
\end{figure}

For each root $\gamma \in \Phi$, let $U_{\gamma}$ be the associated root subgroup in $\G_2$ and fix an isomorphism $x_{\gamma} : F \rightarrow U_{\gamma}$. 
For $g_1,g_2 \in \G_2$, let $[g_1, g_2] = g_1^{-1}g_2^{-1}g_1 g_2$.
For all $x,y \in F$, we have the following commutator relations (see, for instance, \cite[pp.~443]{ree}),
\begin{align*}
[x_{\alpha}(x),x_{\beta}(y)] & = x_{\alpha+\beta}(-xy)x_{2\alpha+\beta}(-x^2y)x_{3\alpha+\beta}(x^3y)x_{3\alpha+2\beta}(-2x^3y^2) \\
[x_{\alpha}(x),x_{\alpha+\beta}(y)] & = x_{2\alpha+\beta}(-2xy)x_{3\alpha+\beta}(3x^2y)x_{3\alpha+2\beta}(3xy^2)\\
[x_{\alpha}(x),x_{2\alpha+\beta}(y)] & = x_{3\alpha+\beta}(3xy) \\
[x_{\beta}(x),x_{3\alpha+\beta}(y)] & = x_{3\alpha+2\beta}(xy) \\
[x_{\alpha+\beta}(x),x_{2\alpha+\beta}(y)] & = x_{3\alpha+2\beta}(3xy).
\end{align*}
For all remaining pairs of positive roots $\gamma_1, \gamma_2$,  we have $[x_{\gamma_1}(x),x_{\gamma_2}(y)] = e$.

We may realize the group $H \cong \so_4(F)$ as the subgroup generated by $T$ and the images of $x_{\beta}, x_{2\alpha+\beta}$ (since $\so_4$ is chosen to be generated by $\beta$ and $2\alpha+\beta$ -see, for instance, \cite{gelfand1976} in \cite[pp.~137]{Gelfand-collected2}), its Weyl group must be generated by $w_{\beta}$ and $w_{2\alpha+\beta}$. 
Then the Weyl group of $H$ with respect to $T$ is $$W_{\so_4} = \left\{1, w_{\beta}, w_{2\alpha+\beta}, w_{\beta} w_{2\alpha+\beta} \right\}.$$
Let $B_{\so_4}$ be the standard Borel of $H$ with respect to the positive roots $\beta$ and $2\alpha+\beta$, then the set $B\slash B_{\so_4}$ has representatives 
\[
\{x_{\alpha}, x_{\alpha+\beta}, x_{3\alpha+\beta}, x_{3\alpha+2\beta} \}.
\]
For $r_i \in F, i= 1,2,3,4$, write:
$$[r_1, r_2, r_3, r_4] = x_{\alpha}(r_1)x_{\alpha+\beta}(r_2)x_{3\alpha+\beta}(r_3)x_{3\alpha+2\beta}(r_4)$$

\section{Computation of the double cosets representatives}\label{doublecosets}

The set $B \backslash X$ of $B$-orbits in $X$ is finite \cite[Proposition 6.15]{helminckwang1993}; therefore, $B \backslash G / H$ is finite \cite[Corollary 6.16]{helminckwang1993}.  In particular, for any standard parabolic subgroup of a (p-adic) reductive group $G$, the set $P \backslash G / H$ is a finite set. 
\newline
Let $P_{\alpha} = M_{\alpha}N_{\alpha}$ (respectively $P_{\beta}=M_{\beta}N_{\beta}$) be the standard parabolic subgroup of $G$ with Levi factor $M_{\alpha}$ and unipotent radical $N_{\alpha}$ such that $\ran(x_{\alpha}) \subseteq M_{\alpha}$ (respectively $\ran(x_{\beta}) \subseteq M_{\beta})$.
Then $N_{\alpha}$ is generated by the images of $\{x_{\beta}, x_{\alpha+\beta}, x_{2\alpha+\beta}, x_{3\alpha+\beta}, x_{3\alpha+2\beta}\}$ (respectively $N_{\beta}$ is generated by the images of $\{x_{\alpha}, x_{\alpha+\beta}, x_{2\alpha+\beta}, x_{3\alpha+\beta}, x_{3\alpha+2\beta}\}$).
We follow a method implemented by Ginzburg in \cite{Ginzburg} to compute the double cosets representatives of $P_{\beta}$. 
\newline


\begin{lem} \label{representatives}
Let $w_0$ denotes the element $w_{\alpha} w_{\beta} w_{\alpha} w_{\beta} w_{\alpha}$, and let \newline
$r_3 \in F^{\times}\slash F^{\times}{}^2$. The set of representatives of $P_{\beta} \backslash G_2 \slash \so_4$ is:
\begin{multline*}
\{e, w_{\alpha}, w_{\alpha} x_{\alpha}(1), w_{\alpha} w_{\beta} w_{\alpha} x_{3\alpha+\beta}(1), w_{\alpha} w_{\beta} w_{\alpha} x_{\alpha}(1), w_{\alpha} w_{\beta} w_{\alpha} x_{3\alpha+\beta}(1)x_{\alpha}(1),\\
 w_0, w_0x_{\alpha+\beta}(1), w_0x_{3\alpha+2\beta}(1), w_0x_{\alpha+\beta}(1)x_{3\alpha+2\beta}(1), w_0x_{\alpha+\beta}(1)x_{3\alpha+\beta}(r_3) \}
\end{multline*}
\end{lem}

\begin{proof}
The set of representatives for $P_{\beta} \backslash G_2 \slash B$ is
$$A= \{e, w_{\alpha}, w_{\alpha} w_{\beta}, w_{\alpha} w_{\beta} w_{\alpha}, w_{\alpha} w_{\beta} w_{\alpha} w_{\beta}, w_0= w_{\alpha} w_{\beta} w_{\alpha} w_{\beta} w_{\alpha}\}$$
Notice that we have used that the last element in $W_{G_2}$ has order two hence is equal to the other order two element whose action is the same on all roots : $w_{\alpha} w_{\beta} w_{\alpha} w_{\beta} w_{\alpha} w_{\beta}= w_{\beta} w_{\alpha} w_{\beta} w_{\alpha} w_{\beta} w_{\alpha}$. 
The set $B\slash B_{\so_4}$ is $$\{x_{\alpha}(r_1), x_{\alpha+\beta}(r_2), x_{3\alpha+\beta}(r_3), x_{3\alpha+2\beta}(r_4)\}$$
A complete set of representatives of $P_{\beta} \backslash G_2 \slash B_{\so_4}$ is given by: 
$$\mathcal{S}:= \{w[r_1, r_2, r_3, r_4], w \in A, r_i \in F \}$$

In the subsequent step, we will use two tricks to find equivalences between different elements of $\mathcal{S}$:
\begin{itemize}
\item We will rescale the unipotent element from $r_i$ to 1 using a torus element. If $r_1 \neq 0$, we can find a torus element $t$ such that $x_{\alpha}(r_1) = tx_{\alpha}(1)t^{-1}$; since $w_{\alpha} t w_{\alpha}^{-1}$ in $P_{\beta}$ and $t$ in $\so_4$ , we get  $w_{\alpha} x_{\alpha}(r_1) \sim w_{\alpha} x_{\alpha}(1)$. Notice that there also exists a torus element which rescales a product of two root subgroups. 
\item We use the commutator relations given in the previous subsection, along with the expressions given in the Table \ref{table} to simplify the expressions for each $w \in A$.
\end{itemize}

Write $x\sim y$ if $x$ and $y$ are in the same double coset in $P_{\beta} \backslash G_2 \slash \so_4$. 
\newline

Since $x_{\alpha}(r_1)x_{\alpha+\beta}(r_2)x_{3\alpha+\beta}(r_3)x_{3\alpha+2\beta}(r_4)$ belong to $N_{P_{\beta}}$, we have $e.[r_1, r_2, r_3, r_4]  \sim e$,
i.e they are in the same double coset in $P_{\beta} \backslash G_2 \slash \so_4$. 
For instance, consider $w_{\alpha}x_{3\alpha+2\beta}(r_4)x_{3\alpha+\beta}(r_3)x_{\alpha+\beta}(r_2)x_{\alpha}(r_1)$, since $w_{\alpha} x_{3\alpha+2\beta}x_{\alpha+\beta} w_{\alpha}^{-1}$ in $N_{\beta}$, $w_{\alpha} x_{3\alpha+\beta} \in M_{\beta}$, what remains is $w_{\alpha} x_{\alpha}$.
The same logic applies to reduce $w_{\alpha}w_{\beta}x_{3\alpha+2\beta}(r_4)x_{3\alpha+\beta}(r_3)x_{\alpha+\beta}(r_2)x_{\alpha}(r_1)$ to $w_{\alpha} w_{\beta} x_{\alpha+\beta}x_{\alpha}$.
Since $x_{3\alpha+\beta}(1)$ and $x_{\alpha+\beta}(1)$  commute, we obtain $w_{\alpha} w_{\beta} w_{\alpha} x_{3\alpha+\beta}(1)x_{\alpha}(1)$ and we also have  $w_{\alpha} w_{\beta} w_{\alpha} w_{\beta} x_{3\alpha+2\beta}(1)x_{\alpha+\beta}(1)$. The last representative $w_0[r_1, r_2, r_3, r_4]$ will be dealt with in the last part of this proof.

\begin{multline}
\{e, w_{\alpha}, w_{\alpha} x_{\alpha}(1); w_{\alpha} w_{\beta} , w_{\alpha} w_{\beta} x_{\alpha+\beta}(1)x_{\alpha}(1), w_{\alpha} w_{\beta} x_{\alpha+\beta}(1), w_{\alpha} w_{\beta} x_{\alpha}(1) ; \\
 w_{\alpha} w_{\beta} w_{\alpha}, w_{\alpha} w_{\beta} w_{\alpha} x_{3\alpha+\beta}(1)x_{\alpha}(1); w_{\alpha} w_{\beta} w_{\alpha} x_{3\alpha+\beta}(1); w_{\alpha} w_{\beta} w_{\alpha} x_{\alpha}(1); \\
  w_{\alpha} w_{\beta} w_{\alpha} w_{\beta}, w_{\alpha} w_{\beta} w_{\alpha} w_{\beta} x_{3\alpha+2\beta}(1)x_{\alpha+\beta}(1), w_{\alpha} w_{\beta} w_{\alpha} w_{\beta} x_{3\alpha+2\beta}(1), w_{\alpha} w_{\beta} w_{\alpha} w_{\beta} x_{\alpha+\beta}(1),  \\
  w_0x_{\alpha+\beta}(1), w_0x_{3\alpha+2\beta}(1), w_0x_{\alpha+\beta}(1)x_{3\alpha+2\beta}(1), w_0[0, 1, r_3, 0]\}
\end{multline}

The second step in this procedure is to look at these elements, as compared to the set $W_{\so_4}$ and try to simplify further:
$$w_{\alpha} w_{\beta}  \sim w_{\alpha}$$
$$ w_{\alpha} w_{\beta} x_{\alpha+\beta}(1)x_{\alpha}(1) \sim w_{\alpha} x_{\alpha}(1) w_{\beta} x_{\alpha}(1) \sim w_{\alpha} x_{\alpha}(1) x_{\alpha+\beta}(1) w_{\beta} \sim w_{\alpha} x_{\alpha}(1) x_{\alpha+\beta}(1)$$
$$ w_{\alpha} w_{\beta} x_{\alpha+\beta}(1) \sim w_{\alpha} x_{\alpha}(1) w_{\beta} \sim w_{\alpha} x_{\alpha}(1); w_{\alpha} w_{\beta} x_{\alpha}(1) \sim w_{\alpha} x_{\alpha+\beta}(1)w_{\beta} \sim w_{\alpha} x_{\alpha+\beta}(1)$$
$$ w_{\alpha} w_{\beta} w_{\alpha} w_{\beta} \sim w_{\alpha} w_{\beta} w_{\alpha}$$
\begin{multline}
w_{\alpha} w_{\beta} w_{\alpha} w_{\beta} x_{3\alpha+2\beta}(1)x_{\alpha+\beta}(1) \sim w_{\alpha} w_{\beta} w_{\alpha} x_{3\alpha+\beta}(1)x_{\alpha}(1); \\
w_{\alpha} w_{\beta} w_{\alpha} w_{\beta} x_{3\alpha+2\beta}(1) \sim w_{\alpha} w_{\beta} w_{\alpha} x_{3\alpha+\beta}(1); w_{\alpha} w_{\beta} w_{\alpha} w_{\beta} x_{\alpha+\beta}(1) \sim w_{\alpha} w_{\beta} w_{\alpha} x_{a}(1)
\end{multline}

$$w_{\alpha} w_{\beta} w_{\alpha} w_{\beta} w_{\alpha} w_{\beta}  \sim w_{\alpha} w_{\beta} w_{\alpha} w_{\beta} w_{\alpha} = w_0 \in W_{\so_4}$$


$w_{\alpha} x_{\alpha}(1)x_{\alpha+\beta}(1) \sim w_{\alpha} x_{\alpha+\beta}(1)x_{\alpha}(1)x_{2\alpha+\beta}(1)x_{3\alpha+\beta}(1)x_{3\alpha+2\beta}(1)$ \\
$\sim  w_{\alpha} x_{\alpha}(1)x_{3\alpha+\beta}(1)x_{3\alpha+2\beta}(1)x_{2\alpha+\beta}(1)$ since $x_{2\alpha+\beta}(1)$ is in $\so_4$ it disappears. 
\noindent
We are left with $w_{\alpha} x_{3\alpha+\beta}(1)x_{3\alpha+2\beta}(1)x_{\alpha}(1)$, and therefore $\cong w_{\alpha} x_{\alpha}(1)$.
$w_{\alpha} w_{\beta} w_{\alpha} \sim w_{\beta} w_{\alpha}^2 w_{\beta} w_{\alpha} w_{\beta} w_{\alpha} \sim w_{\beta} w_{\alpha} w_0 \sim w_{\beta} w_{\alpha}$ since $w_0$ is in $W_{\so_4}$.
\newline

Consider, finally, the representative $w_0[r_1, r_2, r_3, r_4]$. This one cannot be simplified using the tricks described above. However, one notices the $\so_4$ contains a copy of $GL_2$ (constituted of the $x_{\pm \beta}$ and the torus) which commutes with $w_0$. Looking at this representative in the quotient by $\so_4$ gives an action of $GL_2$ on $x_{\alpha}(r_1)x_{\alpha+\beta}(r_2)$ which is the standard action of $GL_2$ on a two-dimensional vector space. Under this action, there are two orbits, one with $r_1=r_2=0$ and the second where $(r_1,r_2)\neq (0,0)$. The first orbit yields the representative $w_0[0, 0, r_3, r_4]$ which, by an action of the same $GL_2$ on the two-dimensional vector space generated by $x_{3\alpha+\beta}(r_3)x_{3\alpha+2\beta}(r_4)$  
yields two representatives $w_0$, and $w_0[0, 0, 0, 1]$.
\newline

For the second orbit, $(r_1,r_2)\neq (0,0)$, we may assume without loss of generality, that $(r_1,r_2)= (0,1)$, then we are reduced to $w_0[0, 1, r_3, r_4]$. Now, either $r_3 = r_4 = 0$, which yields the representative $w_0[0, 1, 0, 0]$; or $r_3 = 0$ and $r_4 \neq 0$, in which case, you can choose a torus element in $\so_4$ which acts linearly on $x_{2\alpha +3\beta}(r_4)$ and commutes with $x_{\alpha+\beta}(r_2)$ so we can reduce further the expression to $w_0[0, 1, 0, 1]$. \\
Finally, if $r_3 \neq 0$, one first conjugates by a suitable element of the form $x_{\beta}(m)$ the expression $w_0[0, 1, r_3, r_4]$ to obtain $w_0[0, 1, r_3, 0]$ (this is easily checked in \textit{SageMath}, one should obtain $m=-r_3/r_4$) , and further there exists an element of the torus $t_1$ such that $x_{3\alpha+\beta}(r_3)x_{\alpha+\beta}(1) =t_1 x_{3\alpha+\beta}(1)x_{\alpha+\beta}(1) t_1^{-1}$ (more specifically this torus satisfies $s=1$, $t^2=r_3$). 
\newline

Then observe that the torus which commutes with $x_{\alpha+\beta}(1)$ (i.e, you can check that too, it requires $s=t$) acts by a square on $x_{3\alpha+\beta}(r_3)$. Therefore this representative becomes $w_0[0, 1, r_3, 0]$ where $r_3 \in F^{\times}\slash F^{\times}{}^2$. 
To show that there a finite number of such representatives, one just needs to recall that when $F$ is a local field, $F^{\times}\slash F^{\times}{}^2$ is finite.  More specifically, let us denote $\pi$ a prime in $F$ a local field, $U = \of^{\times}$ and $U_1 = \left\{1+ x\pi^n | x \in \of\right\}$, and let us take $u$ an element of $U$ with the property that its image in $U/U_1$ is not a square. If $2\nmid q$ then $\left\{1, u, \pi, \pi u \right\}$ form a complete set of cosets representatives for $F^{\times}\slash F^{\times}{}^2$.
\end{proof}

\begin{rmk}
The reader may have noticed that this set is pretty large (eleven representatives!) whereas we would expect its dimension to be really smaller (already it is known that $\dim(G_2/P_{\beta})$ should be 5). The reference \cite{Ginzburg} also uses further simplifications by allowing root subgroups of negative roots (other than $x_{-\alpha}$ or $x_{-\beta}$) to appear in the simplifications. The reason why we have not allowed those root subgroups of negative roots to appear is due to our embedding in $GL_8$ and the fact that we would therefore need explicit embeddings of those root subgroups in $GL_8$ to proceed with further computations in \textit{SageMath}. But the results of the Appendix do not tell us how to express them in $GL_8$. 
\end{rmk}

\section{Analysis of the orbits} \label{analysisorb}
\subsection{Involutions} \label{involutions}

In Section \ref{sec-g2}, we have shown that our involution was defined to be the conjugation by an order two element which was chosen to be a torus element of order two. Let us define three such elements, and take $\theta_{t_i}$ to denote the corresponding involution $\Int(t_i)$ on $G_2$ whose fixed points are $\so_4$: 

\begin{flushleft}
\begin{small}
$$t_0=\left(\begin{array}{rrrrrrrr}
1 & 0 & 0 & 0 & 0 & 0 & 0 & 0 \\
0 & -1 & 0 & 0 & 0 & 0 & 0 & 0 \\
0 & 0 & 1 & 0 & 0 & 0 & 0 & 0 \\
0 & 0 & 0 & -1 & 0 & 0 & 0 & 0 \\
0 & 0 & 0 & 0 & 1 & 0 & 0 & 0 \\
0 & 0 & 0 & 0 & 0 & -1 & 0 & 0 \\
0 & 0 & 0 & 0 & 0 & 0 & 1 & 0 \\
0 & 0 & 0 & 0 & 0 & 0 & 0 & -1
\end{array}\right)
~~
t_1=\left(\begin{array}{rrrrrrrr}
1 & 0 & 0 & 0 & 0 & 0 & 0 & 0 \\
0 & 1 & 0 & 0 & 0 & 0 & 0 & 0 \\
0 & 0 & 1 & 0 & 0 & 0 & 0 & 0 \\
0 & 0 & 0 & 1 & 0 & 0 & 0 & 0 \\
0 & 0 & 0 & 0 & -1 & 0 & 0 & 0 \\
0 & 0 & 0 & 0 & 0 & -1 & 0 & 0 \\
0 & 0 & 0 & 0 & 0 & 0 & -1 & 0 \\
0 & 0 & 0 & 0 & 0 & 0 & 0 & -1
\end{array}\right)$$
and 
$$t_2= \left(\begin{array}{rrrrrrrr}
1 & 0 & 0 & 0 & 0 & 0 & 0 & 0 \\
0 & -1 & 0 & 0 & 0 & 0 & 0 & 0 \\
0 & 0 & -1 & 0 & 0 & 0 & 0 & 0 \\
0 & 0 & 0 & 1 & 0 & 0 & 0 & 0 \\
0 & 0 & 0 & 0 & -1 & 0 & 0 & 0 \\
0 & 0 & 0 & 0 & 0 & 1 & 0 & 0 \\
0 & 0 & 0 & 0 & 0 & 0 & 1 & 0 \\
0 & 0 & 0 & 0 & 0 & 0 & 0 & -1
\end{array}\right)$$
\end{small}
\end{flushleft}

As our reader may also be taking \cite{offen2017} as a reference, and since we are dealing with the definition of involution, we show that $\tau$ (used in this reference, p213) is trivial. Recall that the set of minimal semi-standard parabolic $P_0$ subgroups of a reductive group $G$ forms a $W$-torsor. In particular, since $T$ is $\theta$-stable, there exists a unique Weyl element $\tau \in W$ such that $\theta(P_0) = \tau P_0 \tau^{-1}$. Applying $\theta$ to this identity yields also the condition $\theta(\tau)\tau = e$.

\begin{prop} \label{tau}
Let $\tau$ be the unique Weyl element $\tau \in W$ such that $\theta(P_0) = \tau P_0 \tau^{-1}$. then $\tau = e$ for any $\theta_{t_i}$ where $\theta_{t_i}$ is the involution defined as the conjugation by the order two element $t_i$.
\end{prop}

\begin{proof}
This is clear once one notices that the Borel subgroup $B=P_0$ is $\theta$-stable. See the results of the computation in \textit{SageMath}, file \enquote{tau-p213}. 
\end{proof}

\subsection{The matching} \label{matching}
In this paper, our investigation now focuses solely on the case of $P=P_{\beta}$. Let us denote $W_{\beta} := W_{M_{\beta}} = <w_{\beta}>$.
Following Lemma 3.1 in \cite{offen2017}, we know each orbit representative $\eta$, as given in the previous section, corresponds to a unique element in the double cosets space $W_{\beta}\backslash W\slash W_{\beta} = \left\{e, w_{\alpha}, w_{\alpha}.w_{\beta}.w_{\alpha}, w_{\alpha}.w_{\beta}.w_{\alpha}.w_{\beta}.w_{\alpha} \right\}$. 
\newline

Recall the following map from Subsection \ref{admcondition}: 
$$\iota_M : P\backslash X \rightarrow {}_{M}{W}{_{M'}}\tau^{-1} \cap \mathcal{S}_0(\theta),$$
where $\mathcal{S}_0(\theta) = \{ w \in W : w\theta(w) = e\}$ is the set of twisted involutions in the Weyl group. Here $M'$ is the $\theta' = \theta$ conjugate of $M$.
In our context, first the set of twisted involutions is just the set of involutions, as our involution consists in the conjugation by an order two element of the torus, secondly out of the twelve elements in $W$, seven are indeed involutions. This is easily verified with \textit{SageMath}, although our readers need to pay attention that the product $ww^{-1}$ might not necessarily be the identity matrix, but can also be an order two element of the torus.

Fix $x \in X$, and recall $x = \eta. e = \eta e \theta(\eta)^{-1}$.
$\eta, x$, match some unique elements in the double cosets $W_{\beta}\backslash W \slash W_{\beta}$: 
$w = \iota_{M}(P \dot x)$. This uniqueness follows from a statement at the bottom of p216 in \cite{offen2017} and Proposition \ref{tau}. Offen uses expressions which depend on $P'$ and $M'$ but since $\theta_{t_0}$ stabilizes $M$, $M'=M$ as we have verified this matching using $\theta_{t_0}$. 
\newline

Our first step while dealing with this project was to verify this matching and to do so we have used the involution given by $\theta_{t_0}$ (this verification was not done for $\theta_{t_1}$ or $\theta_{t_2}$). Concretely, we are verifying in \textit{SageMath} (again the code is available in the github file for the convenience of the reader) the following equations.
\begin{itemize}
\item $P x P = P w P$
\item $t_0 = w*t_0*w$
\item  $w$ is left and right $W_{(M_{\beta})}$-reduced.
\end{itemize}
Look at the first point above in \emph{SageMath}: we compare each left side of the equation to the four elements in $W_{\beta} \backslash W \slash W_{\beta}$ and eliminate progressively variables to reach some contradiction for all elements but one which is the match. The results are given in the bracket below: 
\newline

\begin{small}
\begin{equation*}\label{etax}
\left\{
   \begin{array}{lll} 
   \eta & x= \eta.\theta(\eta)^{-1} & W_{\beta}\backslash W\slash W_{\beta} \\
         e & e & e\\
         w_{\alpha} & w_{\alpha}.w_{\alpha} & e\\
         w_{\alpha} x_{\alpha}(1) & w_{\alpha}.x_{\alpha}(2).w_{\alpha} & w_{\alpha}\\
         w_{\alpha} w_{\beta} w_{\alpha} x_{3\alpha+\beta}(1) &w_{\alpha}*w_{\beta}*w_{\alpha}*x_{3\alpha+\beta}*x_{3\alpha+\beta}*w_{\alpha}*t_0*w_{\beta}^{-1}*t_0*w_{\alpha} & w_{\alpha}.w_{\beta}.w_{\alpha}\\
         w_{\alpha} w_{\beta} w_{\alpha} x_{\alpha}(1) &  w_{\alpha}*w_{\beta}*w_{\alpha}*x_{\alpha}*x_{\alpha}*w_{\alpha}*t_0*w_{\beta}^{-1}*t_0*w_{\alpha} & w_{\alpha}.w_{\beta}.w_{\alpha}.w_{\beta}.w_{\alpha} \\
         w_{\alpha} w_{\beta} w_{\alpha} x_{3\alpha+\beta}(1)x_{\alpha}(1) & & w_{\alpha}.w_{\beta}.w_{\alpha}.w_{\beta}.w_{\alpha} \\
         w_0 & w_0t_0w_0^{-1}t_0 & w_0 \\
         w_0x_{\alpha+\beta}(1) &  & w_{\alpha}\\
          w_0x_{3\alpha+2\beta}(1) & & w_{\alpha}.w_{\beta}.w_{\alpha} \\
          w_0x_{\alpha+\beta}(1)x_{3\alpha+2\beta}(1)& &  w_{\alpha}.w_{\beta}.w_{\alpha}\\
          w_0[0, 1, r_3, 0] & & w_{\alpha}.w_{\beta}.w_{\alpha}.w_{\beta}.w_{\alpha}\\
         \end{array}
\right.
\end{equation*}
\end{small}
\subsection{Conventions for the torus and the Levi subgroups} \label{convention}
For the rest of this paper let us set the following notation $\GL_2 = \GL_2(F)$.

Let $t$ and $s$ be $F$-variables. There exist two conventions to write the torus in $M_{\beta}$ in the literature (see for instance \cite{muicreg2} and \cite{ZhangFF}). From the Appendix \ref{sec-g2-apdx} which defines the embedding of $\G_2$ into $\GL_8$ (see in particular the Equation A.1) we are writing the torus in $M_{\beta}$, as $T_{\GL_2}= \left(\begin{array}{rr}s & 0  \\0 & ts^{-1}\end{array}\right)$, so that $\beta(T_{\GL_2}) = e_1 - e_2 = s^2t^{-1}$. 
Therefore, again by \ref{sec-g2-apdx}, the embedding of the torus $\left(\begin{array}{rr}s & 0  \\0 & ts^{-1}\end{array}\right)$ of $\G_2$ in $\GL_8$, is the following:
$$T = T_{\GL_8} = \left(\begin{array}{rrrrrrrr}
1 & 0 & 0 & 0 & 0 & 0 & 0 & 0 \\
0 & \frac{s^{2}}{t} & 0 & 0 & 0 & 0 & 0 & 0 \\
0 & 0 & \frac{t}{s^{2}} & 0 & 0 & 0 & 0 & 0 \\
0 & 0 & 0 & 1 & 0 & 0 & 0 & 0 \\
0 & 0 & 0 & 0 & \frac{t}{s} & 0 & 0 & 0 \\
0 & 0 & 0 & 0 & 0 & s & 0 & 0 \\
0 & 0 & 0 & 0 & 0 & 0 & \frac{1}{s} & 0 \\
0 & 0 & 0 & 0 & 0 & 0 & 0 & \frac{s}{t}
\end{array}\right)$$

Since most of our computations are implemented in \textit{SageMath}, we need an explicit computable expression of the Levi $M= M_{\beta}$.
We use the Bruhat decomposition to consider $M_{\beta}$ as the disjoint union of the two Bruhat cells :$B.w_{\beta}.\overline{U_{\beta}}$ and $B.e.\overline{U_{\beta}}$, written in \textit{SageMath} as: $U_{\beta}TU_{-b}w_{\beta}$ and $TU_{\beta}U_{-\beta}$. Let $m$ be the $F$-variable entering in the matrix expression of $U_{\beta}$ and $x$ be the one used in $U_{-\beta}$, then the two cells are:

\begin{footnotesize}
\begin{equation} \label{bruhatcells}
\left(\begin{array}{rrrrrrrr}
1 & 0 & 0 & 0 & 0 & 0 & 0 & 0 \\
0 & \frac{t}{s} & 0 & 0 & \frac{t x}{s} & 0 & 0 & 0 \\
0 & 0 & \frac{b s x}{t} + \frac{s}{t} & 0 & 0 & 0 & 0 & -\frac{b s}{t} \\
0 & 0 & 0 & 1 & 0 & 0 & 0 & 0 \\
0 & b s & 0 & 0 & b s x + s & 0 & 0 & 0 \\
0 & 0 & 0 & 0 & 0 & t & 0 & 0 \\
0 & 0 & 0 & 0 & 0 & 0 & \frac{1}{t} & 0 \\
0 & 0 & -\frac{x}{s} & 0 & 0 & 0 & 0 & \frac{1}{s}
\end{array}\right) ~~
\left(\begin{array}{rrrrrrrr}
1 & 0 & 0 & 0 & 0 & 0 & 0 & 0 \\
0 & \frac{t x}{s} & 0 & 0 & -\frac{t}{s} & 0 & 0 & 0 \\
0 & 0 & -\frac{b}{s} & 0 & 0 & 0 & 0 & -\frac{b x}{s} - \frac{s}{t} \\
0 & 0 & 0 & 1 & 0 & 0 & 0 & 0 \\
0 & \frac{b t x}{s} + s & 0 & 0 & -\frac{b t}{s} & 0 & 0 & 0 \\
0 & 0 & 0 & 0 & 0 & t & 0 & 0 \\
0 & 0 & 0 & 0 & 0 & 0 & \frac{1}{t} & 0 \\
0 & 0 & \frac{1}{s} & 0 & 0 & 0 & 0 & \frac{x}{s}
\end{array}\right)
\end{equation}
\end{footnotesize}

Let $a,b,c,d,T,u,v,w,X$ be $F$-variables.
The following matrix would make an instance of $M_{\beta}$ since it can be either of the two cells:

\begin{equation} \label{mb}
\left(\begin{array}{rrrrrrrr}
1 & 0 & 0 & 0 & 0 & 0 & 0 & 0 \\
0 & T & 0 & 0 & a & 0 & 0 & 0 \\
0 & 0 & u & 0 & 0 & 0 & 0 & b \\
0 & 0 & 0 & 1 & 0 & 0 & 0 & 0 \\
0 & c & 0 & 0 & v & 0 & 0 & 0 \\
0 & 0 & 0 & 0 & 0 & w & 0 & 0 \\
0 & 0 & 0 & 0 & 0 & 0 & \frac{1}{w} & 0 \\
0 & 0 & d & 0 & 0 & 0 & 0 & x
\end{array}\right)
\end{equation}

\subsection{Admissible orbits}\label{sec-adm-orb}

From now on, we are numbering the elements as they have been ordered in the Equation \ref{etax}: for instance we may write the \enquote{fifth element} to mean $x_5$ as given in the fifth line of that brace (the \textbf{same numbering} is also used in the github code). 
\newline

To apply Proposition 7.1 and 7.2 of \cite{offen2017}, we need the condition $x \in N_{G,\theta}(M)$ to hold. Notice that when we choose $\theta_{t_0}$, this condition is just $M_x= M$. Further, the conditions of openness or closedness of parabolic orbits is verifiable by looking at $\theta_x(P)$ and $P$: either there are equal (closed orbit) either their intersection is the Levi $M$. In the github code we have computed (with the two Bruhat cells given in \ref{bruhatcells}, but also with the expression \ref{mb}, and with the opposite Levi) these conditions. In this subsection, we write all the exact properties verified by the orbits that our computational strategy allowed us to prove. It is not excluded that other orbits could be shown to be admissible, open or closed using either a different involution or strategy.  

\subsection*{The fifth element}
We assume $\theta = \theta_{t_2}$ in this subsection, then $x_5= \eta_5\theta(\eta_5)^{-1}$ is $M$-admissible. The intersection $\theta_x(P)$ and $P$ is equal to $M.U_{2\alpha+\beta}(n)$ (for a $F$-variable $n$), which is slightly larger than $M$ and therefore we may assume that the associated orbit is not open. It is neither closed. 

\subsection*{$\eta_2$ and $\eta_7= w_0$}

\begin{prop} \label{closedorbits}
Let $\eta$ be either $\eta_2= w_{\alpha}$ or $\eta_7= w_0$, then $x_2$ (resp. $x_7$) is $M$-admissible and the orbit is closed. Further $M_x = M$ and $U_x = U$.
\end{prop}

\begin{proof}
First we notice that $x_2 = t_2$ and $x_7 = t_0$, by choosing $\theta = \theta_{t_i}$ accordingly, we observe that $\theta_{t_i}(M) = M$ and $\theta_{t_i}(P) = P$ (see the github code). We further calculate that $U_x = U$. We also notice that these elements are in the second orbit as given in Lemma \ref{disjointunion}. Notice also that the condition $x \in L.w$ is obviously satisfied for $x_7$ since it is a torus element.

In the case of $t_0$, since $\theta_{t_0}$ stabilizes $M$, we further notice that, $L = M \cap xMx^{-1}$ (see page 2 of \cite{offen2017}), hence $L=M$. The conjugation by $t_2$ does not fix the Levi, however we check that $M\eta_2 t_2 \eta_2 = \eta_2 t_2 \eta_2 M$, therefore $L = M$ again. 
\newline

Further, Lemma 6.3 (\textit{ibid}) shows that whenever $x \in X\cap N_{G,\theta}(M), ~~ \hbox{then} ~~  P_x= M_x \rtimes U_x$. Obviously then, the modular character $\delta_{P_x}$ is just $\delta_{P_{\beta}}$.
\end{proof}

Notice that the orbit of $x_7$ plays the role of the identity element's orbit in Lemma \ref{disjointunion}. 

\begin{rmk}
Let us notice that $x_2$ is neither admissible nor closed when $t_i = t_0$ (at least using this definition of closeness), whereas it is admissible with $t_1$ and $\theta_x(P) \cap P = MU_{3\alpha +\beta}U_{3\alpha +2\beta}$ (which should mean it is neither closed nor open).
\end{rmk}

\begin{prop} \label{adm}
Let us consider the involutions given by conjugation with the $t_0, t_1, t_2$ as defined in Subsection \ref{involutions}. The only strictly-admissible orbits are the one of the elements $x_7$ and $x_{10}$ with $\theta_{t_0}$, $x_2$ with $\theta_{t_1}$, and $\theta_{t_2}$ and $x_5 = w_{\alpha}w_{\beta}w_{\alpha}u_{\alpha}(1)$ with $\theta_{t_2}$.
\end{prop}

\begin{proof}
We checked the equations of strict-admissibility (see \ref{strictadm}) for all the orbits' representatives in the code "admissibility-openess-closed.ipynb". The reader can read the result after running the relevant codes. 
\end{proof}

\begin{prop} \label{amazing}
The stabilizer of the representative $w_0[0, 1, r_3, 0]$ in $\so_4$ is isomorphic to one its subgroup $\so_2$ and is therefore of minimal dimension. Therefore $P_{\beta} w_0[0, 1, r_3, 0] \so_4$ is open in $\\G_2$, and $\mathcal{O}_{w_0[0, 1, r_3, 0]}$ is an open orbit. The different $\so_2$ in $\so_4$ are parametrized by the square classes $F^{\times} \slash (F^{\times})^2$ and each gives rise to a given open orbit. 
\end{prop}

\begin{proof}
There exists a torus element $t_e(s,ts^{-1})$ in $T_{\beta}$ with $t=s$ which satisfies the following equation: $t_e x_{\alpha+\beta}(1)x_{3\alpha+\beta}(r_3)t_e^{-1} = x_{\alpha+\beta}(1)x_{3\alpha+\beta}(1)$ and $w_0t_e x_{\alpha+\beta}(1)x_{3\alpha+\beta}(r_3)t_e^{-1} = w_0 t_e w_0^{-1} w_0x_{\alpha+\beta}(1)x_{3\alpha+\beta}(r_3)t_e^{-1}$ with $w_0 t_e w_0^{-1}$ denoted $t'$ (which is some element in the torus only depending on the variable $t$) of the torus $T_{\beta}$ and since we look at $w_0x_{\alpha+\beta}(1)x_{3\alpha+\beta}(r_3)$ in a double coset, multiplying on the left by $t_e^{-1}$ and on the right by $t'^{-1}$ is harmless, so $w_0t_e x_{\alpha+\beta}(1)x_{3\alpha+\beta}(r_3)t_e^{-1} \sim t't_e^{-1} w_0 x_{\alpha+\beta}(1)x_{3\alpha+\beta}(r_3) t_et'^{-1}$ but also $\sim w_0 x_{\alpha+\beta}(1)x_{3\alpha+\beta}(1)$. \\
So we have found an element in the torus which stabilizes the orbit $w_0[0, 1, r_3, 0]$ and depends on only one variable (i.e is of dimension one). Further, we notice that $t_e$ acts as square $x_{3\alpha+\beta}(r_3)$. The square class $r_3t^2$ is the quadratic form $e \rightarrow r_3N(e)$, with $e \in E$, with $r_3$ not a square, attached to $E$ the quadratic extension of $F$. We let $V = E+ (-E)$, where $(-E)$ is the same vector space with the negative quadratic form, be the split ambient non-degenerate 4-dimensional quadratic space and $W$ a two-dimensional quadratic subspace of $V$ such that $\so(E) \cong \so(W) \subset \so(V) \cong \so_4$. Therefore, they are as many $\so(W)$ where , as they are quadratic extensions of $F$ (see also the end of the proof of Lemma \ref{representatives}). Each being the stabilizer of minimal dimension (dimension 1) of $w_0[0, 1, r_3, 0]$ gives rise to an open orbit $P \eta H$. 
\end{proof}

\section{$\so_4$-distinguished induced representations of $G_2$}\label{sec-results}

As we are approaching our final results, one question remains untouched: The question of which characters $\chi$ of $\so_4(F)$ are we using when we apply the Proposition \ref{offen7.1}, and how can they be seen, at first, as characters of the Levi $M_{\beta}$ isomorphic to $\GL_2$.

\subsection{Characters of $\so_4$ and characters of $\GL_2$}
Let us recall here how $\GL_2 = \GL_2(F)$ sits inside $\so_4$. $\GL_2\times \GL_2$ operates on $X=\mathrm{M}_2(F)$ by left and right regular representations, preserving determinant (a quadratic form on the 4-dimensional space $X$) up to scalars:
$$(g,h)X= gXh^{-1}$$
Therefore, the subgroup $H=G[\SL_2\times \SL_2] = \left\{ (g,h) \in \SL_2\times\SL_2| \det(g)=\det(h) \right\}$ lies inside $\so_4$, hence there are a few options for $\GL_2$ to be viewed in $\so_4$. 

\begin{lem} \label{charso4}
We assume the characteristic of $F$ is different from 2. The characters of $\so_4(F)$ are the characters of $F^{\times}\slash F^{\times}{}^2$, i.e are quadratic characters of $F^{\times}$.
\end{lem}

\begin{proof}
The characters of $\so_4$ come from the spinor norm as discussed by Serre in \cite{galoiscoh}, 3.2 b). Let $q$  be a nondegenerate quadratic form of rank $n$. There exists a cohomology exact sequence $$\Spin_q(F) \rightarrow \so_q(F) \rightarrow F^{\times}\slash F^{\times}{}^2\rightarrow H^1(\Spin_q, F) \rightarrow H^1(\so_q, F) \rightarrow \hbox{Br}_2(F)$$ 
We have $H^1(\Spin_q, F) = 0$ by a result of Kneser \cite{kneser} (see also \cite{galoiscoh} 3.1) and using the fact that $\Spin_q$ is simply connected. Therefore, we have the sequence $\Spin_q(F) \rightarrow \so_q(F) \rightarrow F^{\times}\slash F^{\times}{}^2\rightarrow 0$. Since $\Spin_q(F)$ is its own commutator, it has no non-trivial characters, and therefore any complex character of $\so_q(F)$ must be trivial on $\Spin_q(F)$. Finally, by the sequence above, it means any complex character of $\so_q(F)$ can be identified with a character of $F^{\times}\slash F^{\times}{}^2$ (a set of cardinality 4 by the arguments recalled at the end of Lemma \ref{representatives}), i.e with a quadratic character of $F^{\times}$.
\end{proof}

\begin{thm}[Closed orbit] \label{closedorbit}
Let $\chi$ be a character of $\so_4(F)$. It is a quadratic character of $F^{\times}$. It can be seen as a character of $\GL_2$ (those are given by $\chi \circ \det$ for a quasi-character $\chi$ of $F^{\times}$). Let $P_{\beta}$ denote the maximal parabolic corresponding to the root $\beta$. The parabolic induced representations of $G_2$ which are $(\so_4,\chi)$-distinguished include the following representations:
\begin{itemize}
\item The induction from $P_{\beta}$ to $G_2$ of the reducible principal series of $\GL_2$, $I(\chi\delta_{P_{\beta}}^{1/2}|.|^{-1/2}\otimes |.|)$, and the induction of the direct sum of this reducible principal series with any smooth representation of $\GL_2$.
\item The induced representation $I_{P_{\beta}}^{G_2}(\chi\delta_{P_{\beta}}^{1/2})$. 
\item The induced representation $I_{P_{\beta}}^{G_2}(\sigma)$ with $\sigma$ a cuspidal representation of $\GL_2$ whose central character is equal to $(\chi\delta_{P_{\beta}}^{1/2})|_Z$ or the induced representation of the direct sum of such representation with any smooth representation of $\GL_2$.
\item The induced representation $I_{P_{\beta}}^{G_2}(\tau \oplus \chi\delta_{P_{\beta}}^{1/2})$ for $\tau$ any smooth representation of $\GL_2$.
\end{itemize}
\end{thm}

\begin{proof}
Let us first remark that the description of the characters of $\so_4$ results from Lemma \ref{charso4}.
Secondly, we have identified three closed parabolic orbits among the eleven orbits, the one associated to the element $x=e$ (which is, by definition, closed), and the one associated to $x_2$ and $x_7$ (see Proposition \ref{closedorbits}). 
\newline

Applying Proposition \ref{offen7.1} we know that if $\sigma$ is $(M_x, \delta_{P_x}\delta_P^{−1/2}\chi^{\eta^{-1}})$-distinguished then $\ind^G_P(\sigma)$ is $(H, \chi)$-distinguished. Since $L=M \cap \eta \theta (\eta^{-1}M\eta)\eta^{-1}=M$, $M_x= L_x = M \cong \GL_2$ in this case. Notice, also from Proposition \ref{closedorbits}, that $\delta_{Q_x} = \delta_{P_{\beta}}$. We are therefore looking at inducing $\GL_2$-representations (denoted $\sigma$ in in the Proposition \ref{offen7.1}) which are $(\GL_2, \delta_{P_{\beta}}^{1/2}\chi)$-distinguished. 
We then use Proposition \ref{GL2distinctionofPS}.
\end{proof}

\begin{thm}[Distinguished induced parabolic representations and admissible orbits]
We take the involution $\theta$ defining $\so_4(F) = \G_2^{\theta}(F)$ to be of the form $\theta_{t_i}$ for $i \in \{0,1,2\}$ as defined in the Subsection \ref{convention}, and the Levi $M_{\beta}$ as defined in Equation \ref{mb}.
The parabolically induced representations from the parabolic $P_{\beta}$ of $G_2$ distinguished by $\so_4$ whose linear forms arise from admissible, open or closed, orbits are necessarily of the form given in the previous Theorem\ref{closedorbit}.
\end{thm}

\begin{proof}
 We apply the Proposition \ref{offen7.1} and refer to Proposition 7.2 in \cite{offen2017} for the open orbit. Notice that the condition $x \in N_{G,\theta}(M)$ is equivalent to the condition of strict-admissibility as defined in the Definition \ref{strictadm}. 
\newline
First, we have shown that following this definition, the only strictly-admissible elements are $x_7, x_{10}, x_2$ and $x_5$, with only the orbits of $x_2$ and $x_7$ satisfying the closedness condition. In particular, the open orbit (corresponding to $x_{11}$) is shown to be non-admissible. Thus, to apply the Proposition \ref{offen7.1} and Proposition 7.2 in \cite{offen2017}, the following condition is necessarily satisfied: $\theta_x(M) = M$ and $M_x = M \cong GL_2$. In other words, the case where $M_x = T$ does not occur. The case where the orbit is closed, and $M \cong GL_2$ was treated in the Theorem \ref{closedorbit}. 
\end{proof}

\appendix

\section{Conventions for $\G_2$ used for SageMath computations}\label{sec-g2-apdx}

The appendix contains the necessary background information that allows one to embed the exceptional group $\G_2$ into the general linear group $\GL(8)$.  Realizing $\G_2$ as the automorphism group of an eight-dimensional Cayley algebra $\mathcal{C}$, and then computing the matrices of elements in root groups with respect to a chosen basis for $\mathcal{C}$ is enough to produce the embedding.  Such an embedding was used extensively to carry out the calculations in \emph{SageMath} that are used throughout the present paper.
The material in this appendix has been graciously provided by Steven Spallone who in turn would like to
acknowledge Gordan Savin for getting him started.  Any errors in what follows are the responsibility of the author.

\subsection*{Preliminaries on the Cayley algebra}
First, we describe the split Cayley algebra $\mathcal{C}$ over $F$.  Let $\mathrm{M}_2(F)$ be the algebra of $2\times 2$ matrices with entries in $F$. As an $F$-vector space, $\mathcal{C}=\mathrm{M}_2(F) \oplus \mathrm{M}_2(F)$ and a typical element of $\mathcal{C}$ can be  written as a pair $c=(x \mid y)$, where $x,y \in \mathrm{M}_2(F)$. Multiplication on $\mathcal{C}$ is given by
\begin{align*}
(x \mid y)(x' \mid y')=(xx'+ \adj{y'}y \mid y'x+ y \adj{x'}),
\end{align*}
for all $(x \mid y), (x' \mid y') \in \mathcal{C}$.
Here $\adj x$ is the usual adjugate matrix, which agrees with $(\det x) \cdot x^{-1}$ when $x \in \mathrm{M}_2(F)$ is invertible.  The algebra $\mathcal{C}$ has an identity $e=(I_2 \mid 0)$, where $I_2$ is the $2\times 2$ identity matrix, and the subspace spanned by $e$ is the centre of $\mathcal{C}$.

There is a conjugation map on $\mathcal{C}$ given by 
\begin{align*}
\ol{(x \mid y)}=(\adj x \mid -y),
\end{align*}
and a norm map $N: \mathcal{C} \to F$ given by
\begin{align*}
N((x \mid y))=\det x-\det y.
\end{align*}
For any $c \in \mathcal{C}$, the trace of $c$ is defined to be $c+ \ol c$.
If $c = (x \mid y)$, then 
\[
c + \ol{c} = \tr(x) e,
\]
which we identify with the usual trace $\tr(x) \in F$ of $x$.
Thus, we abuse notation and write $\tr: \mathcal{C} \rightarrow F$ for the map $c \mapsto c + \ol{c}$. 
The bilinear form determined by $N$, namely the pairing defined for $c, d \in \mathcal{C}$ by 
\begin{align*}
\ip{c}{d} &= N(c+d)-N(c)-N(d)
\end{align*}
is non-degenerate.  Observe that $\ip{c}{d} =  \tr(c \ol d)$ for all $c,d \in \mathcal{C}$; in particular,  $\ip{c}{e}=\tr(c)$, for all $c \in \mathcal{C}$.

\subsection*{The Automorphism Group of $\mathcal{C}$}
Let $G$ be the group of automorphisms of the algebra $\mathcal{C}$.  It is now well known that $G$ is a split semisimple algebraic group of type $\G_2$  (this was first proved by E.~Cartan \cite{cartan1914}). By \cite{veldkamp}, the elements of $G$ stabilizing $\mathcal{A}_{2} = \left(\begin{array}{cc|cc} 
 *& *  & 0 & 0 \\
 *& *  & 0 & 0 \\
\end{array} \right)$ are of the form $\varphi_{c,p}$, where
\begin{equation*}
\varphi_{c,p}(x \mid y)=(cxc^{-1} \mid pcyc^{-1}),
\end{equation*}
with $c \in \GL_2(F)$ and $p \in \SL_2(F)$.

Let $\lambda_1,\lambda_2 \in F^\times$.  Let $a_{\lambda_1, \lambda_2} \in \GL_2(F)$ be the diagonal matrix 
\begin{align*}
a_{\lambda_1,\lambda_2}= \left( \begin{array}{cc}
\lambda_1 & 0 \\
0 & \lambda_2 \\
\end{array} \right).
\end{align*}
Then define the element $\gamma(\lambda_1,\lambda_2) \in G$ via

\begin{equation}\label{torus}
\gamma(\lambda_1,\lambda_2)(x \mid
 y) = \left( \Int(a_{\lambda_1,\lambda_2})(x)  \mid 
 a_{\lambda_2, {\lambda_2}^{-1}}
\Int(a_{\lambda_1,\lambda_2})(y)
 \right),
\end{equation}
for all $(x \mid y) \in \mathcal{C}$.
Recall that  $\Int(g)(x)=gxg^{-1}$, for any $g\in \GL_{2}(F)$ and $x \in \mathrm{M}_{2}(F)$. 

Let $T=\{ \gamma(\lambda_1,\lambda_2) :  \lambda_1,\lambda_2 \in F^{\times} \}$; then $T$ is a maximal torus of $G$.
Let $\gamma=\gamma(\lambda_1, \lambda_2) \in T$.  Define $\alpha(\gamma)=\lambda_1 {\lambda_2}^{-1}$ and $\beta(\gamma)={\lambda_2}^2 {\lambda_1}^{-1}$.
Then we have
\begin{align*}
(\alpha+\beta)(\gamma)& =\lambda_2, \\
 (2 \alpha+ \beta)(\gamma)& =\lambda_1, \\ 
 (3 \alpha+ \beta)(\gamma)&={\lambda_1}^2 {\lambda_2}^{-1}, & \text{and} \\
  (3 \alpha+ 2 \beta)(\gamma)&=\lambda_1 \lambda_2.
\end{align*}
Let $s=\left( \begin{array}{cc}
0 & -1 \\
1 & 0 \\
\end{array} \right)$ and define $w_G \in G$ by 
\begin{align*}
{w_G}((x \mid y))=(s x s^{-1}  \mid sys^{-1}). 
\end{align*}
Then conjugation by $w_G$ acts by inversion on $T$, and thus represents the longest Weyl group element of $T$ in $G$.

\subsection*{Lie algebra $\mathfrak{g}$}
The Lie algebra $\mathfrak{g}$ of $G$ can be identified with the algebra of derivations of $\mathcal{C}$. Recall that a derivation of $\mathcal{C}$ is a linear map $D: \mathcal{C} \to \mathcal{C}$ so that
\begin{align*}
D(cd)=D(c)d + c D(d),
\end{align*}
for all $c, d \in \mathcal{C}$.

The adjoint action of $G$ on $\mathfrak{g}$ is given by $(\Ad(g)D)(c)=g(D({g^{-1}}c))$, for all derivations $D \in \mathfrak{g}$ and $c\in\mathcal{C}$. Let $\mathfrak t$ denote the Lie algebra of the torus $T$. Let $\lambda_{1},\lambda_{2}\in F$ and let $\gamma(\lambda_{1},\lambda_{2})\in \mathfrak{t}$. Let
\begin{align*}
t= a_{\lambda_1, \lambda_2} = \left( \begin{array}{cc}
\lambda_1 & 0 \\
0 & \lambda_2 \\
\end{array} \right).
\end{align*}

It is easy to see that for an element $(x\mid y)\in \mathcal{C}$ we have
\begin{equation*} 
{}^{\gamma(\lambda_{1},\lambda_{2})}(x \mid y)=([t,x], \tr(t)y-yt).
\end{equation*}

\subsection*{One-parameter root subgroups of $G_{2}$}
\subsubsection*{Root subgroup and other objects related to $\alpha$}
For any $t \in F$, define $u_\alpha(t), u_{-\alpha}(t) \in G$ by 
\begin{align*}
u_\alpha(t)( x \mid y)=( \Int(V(t)) x \mid y V(-t) )
\end{align*}
and
\begin{align*} 
u_{-\alpha}(t)( x \mid y)=(\Int (\ol V(t)) x \mid y \ol V(-t)),
\end{align*}
for all $(x \mid y ) \in \mathcal{C}$, 
where
\begin{align*}
V(t)=\left( \begin{array}{cc} 
 1 & t  \\
 0 & 1  \\
\end{array} \right)\ \ \ \ {\rm and} \ \ \ \ \ol V(t)= \left( \begin{array}{cc} 
 1 & 0  \\
 t & 1  \\
\end{array} \right).
\end{align*}
Explicitly,
\begin{align*}
u_\alpha(t)  \left( \begin{array}{cc|cc} 
 x_1&x_2  & y_1 & y_2 \\
 x_3&x_4  & y_3 & y_4 \\
\end{array} \right)= \left( \begin{array}{cc|cc} 
 x_1+tx_3&x_2 + t(x_4-x_1)-t^2x_3  & y_1 & y_2 -ty_1\\
 x_3&x_4-tx_3  & y_3 & y_4 -ty_3\\
\end{array} \right)
\end{align*}
Let $n_{\alpha}$ be a representative of the reflection in $W_{G_{2}}$ corresponding to the root $\alpha$. Following Chevalley's recipe (see \cite[$\S32.3$]{Humphreys-book}, for instance), one easily computes that $n_\alpha=u_\alpha(1) u_{-\alpha}(-1) u_\alpha(1)$ is given by
\begin{align*}
n_\alpha(x \mid y)=(sxs^{-1} \mid ys)
\end{align*}
where $s=\left( \begin{array}{cc}
0 & -1 \\
1 & 0 \\
\end{array} \right)$, as above. Note that $n_\alpha^2=\gamma(-1,-1)$, and that 
\begin{align*}
n_{\alpha} \gamma(\lambda_1, \lambda_2) n_{\alpha}^{-1}=\gamma(\lambda_2, \lambda_1).
\end{align*}

\subsubsection{Root subgroup and other objects related to $\beta$}

If $a,b \in \mathcal{C}$, define the map $L_{a,b}:\mathcal{C} \rightarrow \mathcal{C}$ by
\begin{align*}
L_{a,b}(c)=\ip{c}{a} b-\ip{c}{b}a,
\end{align*}
for all $c \in \mathcal{C}$.
Take
\begin{align*}
x_0=\left( \begin{array}{cc|cc} 
0 & 0 & 0 & 0 \\
1 & 0 & 0 & 0 \\
\end{array} \right),  
w_0=\left( \begin{array}{cc|cc} 
0 & 0 & 1 & 0 \\
0 &0  & 0 & 0 \\
\end{array} \right),
\end{align*} 
and define $D_{\beta}=L_{w_0,x_0}$. Put
\begin{align*}
x_0'=\left( \begin{array}{cc|cc} 
0 & -1 & 0 & 0 \\
0 & 0 & 0 & 0 \\
\end{array} \right),
w_0'=\left( \begin{array}{cc|cc} 
0 & 0 & 0 & 0 \\
0 & 0 & 0 & 1 \\
\end{array} \right),
\end{align*} 
and define $D_{-\beta}=L_{w_0',x_0'}$. It is straightforward to check that $D_{\beta}$ and $D_{\beta}$ are the root vectors for the roots $\beta$ and $-\beta$ in the Lie algebra $\mathfrak{g}$ of $G$. 
  
Then for $t \in F$ and $c \in \mathcal{C}$, define  
\begin{align*}  
  u_\beta(t)(c)=c+ L_{w_0,x_0}(tc).
  \end{align*}
  and
  \begin{align*}
  u_{-\beta}(t)(c)=c+ L_{w_0',x_0'}(tc).
  \end{align*}
Explicitly,
\begin{align*}
u_\beta(t)    \left( \begin{array}{cc|cc} 
 x_1&x_2  & y_1 & y_2 \\
 x_3&x_4  & y_3 & y_4 \\
\end{array} \right)=  \left( \begin{array}{cc|cc} 
 x_1&x_2  & y_1+tx_2 & y_2 \\
 x_3-ty_4&x_4  & y_3 & y_4 \\
\end{array} \right)
\end{align*}
and
\begin{align*}
u_{-\beta}(t)    \left( \begin{array}{cc|cc} 
 x_1&x_2  & y_1 & y_2 \\
 x_3&x_4  & y_3 & y_4 \\
\end{array} \right)=  \left( \begin{array}{cc|cc} 
 x_1&x_2 +ty_1 & y_1 & y_2 \\
 x_3 & x_4  & y_3 & y_4 -t x_3 \\
\end{array} \right).
\end{align*} 
  
Let $n_{\beta}$ denote the representative of the reflection corresponding to $\beta$ in $W_{G_{2}}$. Following Chevalley's recipe, as above, if we set $n_\beta=u_\beta(1) u_{-\beta}(-1) u_\beta(1)$, then
\begin{align*} 
n_\beta(c)=c+ \ip{c}{w-x'} x + \ip{c}{w'-x}w - \ip{c}{w'+x}x' + \ip{c}{w+x'}w'.
\end{align*}
Explicitly
\begin{align*}
n_\beta   \left( \begin{array}{cc|cc} 
 x_1&x_2  & y_1 & y_2 \\
 x_3&x_4  & y_3 & y_4 \\
\end{array} \right)=
  \left( \begin{array}{cc|cc} 
 x_1&-y_1  & x_2 & y_2 \\
 -y_4&x_4  & y_3 & x_3 \\
\end{array} \right).
\end{align*}
 Note that $n_\beta^2 =\gamma(1,-1)$, and that 
 \begin{align*}
 n_\beta\gamma(\lambda_1, \lambda_2) n_\beta^{-1}=\gamma(\lambda_1,\lambda_1 \lambda_2^{-1}).
  \end{align*}
  
\subsubsection{More root subgroups}
The formula $\Int(n_\alpha)u_{3 \alpha+ \beta}(t)=u_\beta(t)$ gives
\begin{align*}
u_{3 \alpha+ \beta}(t)  \left( \begin{array}{cc|cc} 
 x_1&x_2  & y_1 & y_2 \\
 x_3&x_4  & y_3 & y_4 \\
\end{array} \right)= \left( \begin{array}{cc|cc} 
 x_1&x_2-ty_3  & y_1 & y_2-tx_3 \\
 x_3&x_4  & y_3 & y_4 \\
\end{array} \right).
\end{align*}

The formula $\Int(n_\beta) u_{\alpha+ \beta}(t)=u_\alpha(-t)$ gives
\begin{align*}
u_{\alpha+ \beta}(t)
   \left( \begin{array}{cc|cc} 
 x_1&x_2  & y_1 & y_2 \\
 x_3&x_4  & y_3 & y_4 \\
\end{array} \right)= \left( \begin{array}{cc|cc} 
 x_1+ty_4&x_2  & y_1+t(x_4-x_1)-t^2 y_4 & y_2+tx_2 \\
 x_3+ty_3&x_4-ty_4  & y_3 & y_4 \\
\end{array} \right) .
\end{align*} 

 The formula $\Int(n_\alpha) u_{2\alpha+ \beta}(t)=u_{\alpha+\beta}(-t)$ (see for instance \cite{Humphreys-book}[3.35] gives
\begin{align*}
u_{2\alpha+ \beta}(t)
   \left( \begin{array}{cc|cc} 
 x_1&x_2  & y_1 & y_2 \\
 x_3&x_4  & y_3 & y_4 \\
\end{array} \right)=  \left( \begin{array}{cc|cc} 
 x_1-ty_3&x_2+ty_4  & y_1-tx_3 & y_2+t(x_4-x_1)+t^2y_3 \\
 x_3&x_4+ty_3  & y_3 & y_4 \\
\end{array} \right).
\end{align*} 

The formula $\Int(n_\beta)u_{3\alpha+ 2\beta}(t)= u_{3\alpha+\beta}(-t)$ gives
\begin{align*}
u_{3\alpha+ 2\beta}(t)
   \left( \begin{array}{cc|cc} 
 x_1&x_2  & y_1 & y_2 \\
 x_3&x_4  & y_3 & y_4 \\
\end{array} \right)=  \left( \begin{array}{cc|cc} 
 x_1&x_2  & y_1-ty_3 & y_2-ty_4 \\
 x_3&x_4  & y_3 & y_4 \\
\end{array} \right).
\end{align*} 

\subsection{Embedding $\G_2$ into $\GL(8)$}
The algebra $\mathcal{C}$ is eight-dimensional and has ordered basis $\mathcal{B} = \{e_{11}, e_{21}, e_{31}, e_{41}, e_{12}, e_{22}, e_{32}, e_{42}\}$ where
\begin{align*}
e_{11} & = 	  \left( \begin{array}{cc|cc} 
 1 & 0  & 0 & 0 \\
 0 & 0  & 0 & 0 \\
\end{array} \right) 
& e_{12} & = \left( \begin{array}{cc|cc} 
 0 & 0  & 1 & 0 \\
 0 & 0  & 0 & 0 \\
\end{array} \right) \\
e_{21} & = 	  \left( \begin{array}{cc|cc} 
 0 & 1  & 0 & 0 \\
 0 & 0  & 0 & 0 \\
\end{array} \right) 
& e_{22} & = \left( \begin{array}{cc|cc} 
 0 & 0  & 0 & 1 \\
 0 & 0  & 0 & 0 \\
\end{array} \right) \\
e_{31} & = 	  \left( \begin{array}{cc|cc} 
 0 & 0  & 0 & 0 \\
 1 & 0  & 0 & 0 \\
\end{array} \right) 
& e_{32} & = \left( \begin{array}{cc|cc} 
 0 & 0  & 0 & 0 \\
 0 & 0  & 1 & 0 \\
\end{array} \right) \\
e_{41} & = 	  \left( \begin{array}{cc|cc} 
 0 & 0  & 0 & 0 \\
 0 & 1  & 0 & 0 \\
\end{array} \right) 
& e_{42} & = \left( \begin{array}{cc|cc} 
 0 & 0  & 0 & 0 \\
 0 & 0  & 0 & 1 \\
\end{array} \right).
\end{align*}
Calculating the matrices of $\gamma(\lambda_1, \lambda_2) \in T$, and the root groups described above with respect to the basis $\mathcal{B}$ is enough to embed $G \cong \G_2$ into $\GL(8)$.

\section{The code}
The code is organized in two branches: one \enquote{main} branch where all files needed to justify the results presented in this article are available, the second branch, \enquote{old strategy} (see also the next subsection), containing all the results of our old strategy explained below. 

\subsection*{Details on the codes related to a previous strategy which did not pay off}

This other strategy consisted in:
\begin{itemize}
\item Modifying the elements $x$ to $x' = u . x = ux \theta(u)^{-1}=u x t_0 u^{-1} t_0$. Note that for $x'$, the $w$ (and hence the $L$) remains the same as $x$ (as $PxP= Px'P$). So we need to calculate $u$ such that there exists an element $m \in L$ such that the following equation has a solution $u x t_0 u^{-1} t_0 =mw$ (Note that above we have the freedom to choose $m$ which we can use for our convenience). The reader may be wondering why we chose $u$ from $U$ (instead of a general element in $P$) then that is just for simplicity of computations. Once we know $u$, we have $x'$ and then we will calculate $L_{x'}$, and $U_{x'}$. The computations of $x'$ are available in the github file.
\item Calculating the values of $L$ with the formula $L = M(w\tau) = M \cap w\theta(M)w^{-1}$ page 217 of \cite{offen2017} rather than with the formula page 212: $L= M\cap \eta\theta(\eta^{-1}M\eta)\eta^{-1}$.
\item Calculating the $L_{x'}$ but also $U_{x'}$ using the two previous points.
\end{itemize}
\noindent
It is very likely that the context of working with \emph{SageMath} make all our computations extremely sensitive and modifications as above which could be harmless in another context lead to completely different (and often non-conclusive) results. We, however, decided to include the relevant codes so that the curious reader could explore, and possibly find her way around a similar strategy. Some of the codes written along this path could also be useful to other authors, in particular the delta functions computations.

\vspace{0,2cm}

In this old strategy, we also used the matching results presented in Subsection \ref{matching} and rather worked with $w_{\eta}$ (the unique match of an element $\eta$) in the definition of admissibility and the calculating of $L = M \cap w M w^{-1}$. Since $M \cong \GL_2$, the only subgroups of $M$ are $M$ and the torus. Looking at the possibility of existence of an element in the intersection of $M$ and $w M w^{-1}$, we obtained the following result:

\begin{equation} \label{L}
\left\{
   \begin{array}{lll}
         W_{\beta}\backslash W\slash W_{\beta} & L \\
          e & M \\
         w_{\alpha} & T & ~~\hbox{and}~~ wMw^{-1} \neq w^{-1}Mw\\
         w_{\alpha}.w_{\beta}.w_{\alpha} & T &~~\hbox{and}~~ wMw^{-1} = w^{-1}Mw\\
         w_0 & M & ~~\hbox{and}~~ wMw^{-1} = w^{-1}Mw\\
    \end{array}
\right.
\end{equation}

Although we have explained that the admissibility results were difficult to interpret or misleading using $w_{\eta}$ rather than $x$, we believe the computations of $L$ are confirmed using the more general definition $L= M\cap\eta\theta(\eta^{−1}M\eta)\eta^{−1}$, as given in Theorem 1.1 in \cite{offen2017}. 

\bibliography{g2so4}

\begin{thebibliography}{10}

\bibitem{gelfand1976}
I.~N. Bern\v{s}te\u{\i}n, I.~M. Gel'fand, and S.~I. Gel'fand.
\newblock Models of representations of {L}ie groups.
\newblock {\em Trudy Sem. Petrovsk.}, (Vyp. 2):3--21, 1976.

\bibitem{blanc--delorme2008}
Philippe Blanc and Patrick Delorme.
\newblock Vecteurs distributions {$H$}-invariants de repr{\'e}sentations
  induites, pour un espace sym{\'e}trique r{\'e}ductif {$p$}-adique {$G/H$}.
\newblock {\em Ann. Inst. Fourier (Grenoble)}, 58(1):213--261, 2008.

\bibitem{bump}
Daniel Bump.
\newblock {\em Automorphic forms and representations}, volume~55 of {\em
  Cambridge Studies in Advanced Mathematics}.
\newblock Cambridge University Press, Cambridge, 1997.

\bibitem{cartan1914}
Elie Cartan.
\newblock Les groupes r\'{e}els simples, finis et continus.
\newblock {\em Ann. Sci. \'{E}cole Norm. Sup. (3)}, 31:263--355, 1914.

\bibitem{Casselman-book}
William Casselman.
\newblock Introduction to the theory of admissible representations of $p$-adic
  reductive groups.
\newblock Unpublished manuscript, draft prepared by the S\'{e}minaire Paul
  Sally, 1995.
\newblock Available at
  \href{https://www.math.ubc.ca/~cass/research/publications.html}{www.math.ubc.ca/$\sim$cass/research/publications.html}.

\bibitem{Pseudo-book}
Brian Conrad, Ofer Gabber, and Gopal Prasad.
\newblock {\em Pseudo-reductive groups}, volume~26 of {\em New Mathematical
  Monographs}.
\newblock Cambridge University Press, Cambridge, second edition, 2015.

\bibitem{FLO12}
E.~Feigon, B;~Lapid and O.~Offen.
\newblock On representations distinguished by unitary groups.
\newblock {\em Publ.math.IHES}, 115:185–323, 2012.

\bibitem{gangomez}
W.T Gan and R.Gomez.
\newblock {\em Volume in honor of N. Wallach}, chapter A conjecture of
  Sakellaridis–Venkatesh on the unitary spectrum of spherical varieties.
\newblock 2012.

\bibitem{Gelfand-collected2}
Izrail~M. Gel'fand.
\newblock {\em Collected papers. {V}ol. {II}}.
\newblock Springer-Verlag, Berlin, 1988.
\newblock Edited by S. G. Gindikin, V. W. Guillemin, A. A. Kirillov, B. Kostant
  and S. Sternberg, With a foreword by Kirillov, With contributions by G. Segal
  and C. M. Ringel.

\bibitem{Ginzburg}
David Ginzburg.
\newblock A rankin-selberg integral for the adjoint representation of ${GL}_3$.
\newblock {\em Inventiones mathematicae}, 105(3):571--588, 1991.

\bibitem{helminckwang1993}
A.~G. Helminck and S.~P. Wang.
\newblock On rationality properties of involutions of reductive groups.
\newblock {\em Adv. Math.}, 99(1):26--96, 1993.

\bibitem{Humphreys-book}
James~E. Humphreys.
\newblock {\em Linear algebraic groups}.
\newblock Springer-Verlag, New York-Heidelberg, 1975.
\newblock Graduate Texts in Mathematics, No. 21.

\bibitem{Humphreys1995-book}
James~E. Humphreys.
\newblock {\em Conjugacy classes in semisimple algebraic groups}, volume~43 of
  {\em Mathematical Surveys and Monographs}.
\newblock American Mathematical Society, Providence, RI, 1995.

\bibitem{cech}
S~Jackowski, J~McClure, and B~Oliver.
\newblock Maps between classifying spaces revisited.
\newblock In {\em The \v{C}ech centennial ({B}oston, {MA}, 1993)}, volume 181
  of {\em Contemp. Math.}, pages 263--298. Amer. Math. Soc., Providence, RI,
  1995.

\bibitem{kneser}
M.~Kneser.
\newblock Galoiskohomologie halbeinfacher algebraischer {G}ruppen fiber
  p-adischen {K}orper.
\newblock {\em Math. Zeit}, 1965.

\bibitem{ZhangFF}
B~Liu and Q~Zhang.
\newblock Uniqueness of certain {F}ourier-{J}acobi models over finite fields.
\newblock {\em Finite Fields Their Appl.}, 58:70--123, 2019.

\bibitem{matringe2021}
Nadir Matringe and Omer Offen.
\newblock Intertwining periods and distinction for p-adic galois symmetric
  pairs.
\newblock pre-print, 2021.

\bibitem{Milne2017-book}
J.~S. Milne.
\newblock {\em Algebraic groups}, volume 170 of {\em Cambridge Studies in
  Advanced Mathematics}.
\newblock Cambridge University Press, Cambridge, 2017.
\newblock The theory of group schemes of finite type over a field.

\bibitem{muicreg2}
G.~Mui\'c.
\newblock The unitary dual of ${p}$-adic ${G_2}$.
\newblock {\em Duke Math. J.}, 90:465--493, 1997.

\bibitem{offen2017}
O~Offen.
\newblock On parabolic induction associated with a p-adic symmetric space.
\newblock {\em Journal of Number Theory}, 170:211 -- 227, 2017.

\bibitem{arnabO}
O~Offen and A~Mitra.
\newblock On {S}p-distinguished representations of the quasi-split unitary
  groups.
\newblock {\em Journal of the {I}nstitute of {M}athematics of {J}ussieu}, 2019.

\bibitem{Off06}
Omer Offen.
\newblock {Residual spectrum of ${\rm GL}_{2n}$ distinguished by the symplectic
  group}.
\newblock {\em Duke Mathematical Journal}, 134(2):313 -- 357, 2006.

\bibitem{offenluminy}
Omer Offen.
\newblock {\em Relative {A}spects in {R}epresentation {T}heory, {L}anglands
  {F}unctoriality and Automorphic Forms}, chapter 3, Period Integrals of
  Automorphic Forms and Local Distinction, pages 159--196.
\newblock Springer, 2018.

\bibitem{genericprasad}
Dipendra Prasad.
\newblock Generic representations for symmetric spaces.
\newblock {\em Advances in Mathematics}, 2019.

\bibitem{ree}
Rimhak Ree.
\newblock A family of simple groups associated with the simple lie algebra of
  type (${G_2}$).
\newblock {\em American Journal of Mathematics}, 83(3):432--462, 1961.

\bibitem{galoiscoh}
Jean-Pierre Serre.
\newblock {\em Galois cohomology}.
\newblock Springer-Verlag Berlin Heidelberg, 1997.

\bibitem{veldkamp}
T~Springer and F~Veldkamp.
\newblock {\em Octonions, {J}ordan algebras and exceptional groups. Revised
  English version of the original German notes}.
\newblock Springer, 01 2000.

\bibitem{vigneras}
Marie-France Vignéras; Xuejun~Guo (Translator).
\newblock {\em The {A}rithmetic of {Q}uaternion {A}lgebra}.
\newblock Springer-Verlag Berlin Heidelberg, 1980.

\bibitem{voight}
Voight.
\newblock {\em {Q}uaternion algebras}.
\newblock v0.9.8 October 13, 2017.

\end{thebibliography}

\bibliographystyle{plain} 

%
\end{document}